\documentclass[aos,preprint]{imsart}
\setattribute{journal}{name}{}

\RequirePackage[T1]{fontenc}
\RequirePackage{amsthm,amsmath,amssymb,bbm,mathrsfs,dsfont,wasysym}
\RequirePackage[authoryear]{natbib}
\RequirePackage[colorlinks,citecolor=blue,urlcolor=blue]{hyperref}
\usepackage[french,english]{babel} 
\usepackage[latin1]{inputenc} 
\usepackage[babel=true]{csquotes} 
 
\startlocaldefs
\numberwithin{equation}{section}
\theoremstyle{plain}

\usepackage{needspace}

\theoremstyle{plain}
\newtheorem{lemme}{Lemma}[section]
\newtheorem{proposition}{Proposition}[section]
\newtheorem{theoreme}{Theorem}[section]
\newtheorem{corollaire}{Corollary}[section]

\newtheorem{hypoth}{Assumption}[section]

\newenvironment{hypo}{
    \Needspace*{3\baselineskip}%
    \hypoth
}{\endhypoth}

\theoremstyle{remark}
\newtheorem{remarque}{Remark}[section]

\newtheorem{definition}{Definition}[section]

\endlocaldefs

\newcommand{\R}{\mathbbm{R}}
\newcommand{\N}{\mathbbm{N}}
\newcommand{\J}{\mathrm{J}}
\renewcommand{\H}{\mathrm{H}}
\newcommand{\E}{\mathrm{E}}
\renewcommand{\P}{\mathrm{P}}
\renewcommand{\L}{\mathrm{L}}
\newcommand{\M}{\mathrm{M}}
\renewcommand{\d}{\mathrm{d}}
\newcommand{\trace}{\mathrm{trace}}
\newcommand{\bigo}{\mathcal{O}}

\begin{document}

\begin{frontmatter}
\title{Asymptotically efficient prediction \\for LAN families}
\runtitle{Asymptotically efficient prediction}

\begin{aug}
\author{\fnms{Emmanuel} \snm{Onzon}\ead[label=e1]{emmanuel.onzon@univ-lyon1.fr}},

\runauthor{E. Onzon}

\affiliation{Université Lyon 1}

\address{Bâtiment Polytech Lyon,\\ 15 Boulevard André Latarget,\\
 69100 Villeurbanne, France\\
\printead{e1}\\
\phantom{E-mail:\ }}

\end{aug}

\begin{abstract}
\quad In a previous paper (\cite{bosq2012}) we did a first generalization of the concept of asymptotic efficiency for statistical prediction, \emph{i.e.} for the problems where the unknown quantity to infer is not deterministic but random. 
However, in some instances, the assumptions we made were not easy to verify. Here we give proofs of similar results based on quite a different set of assumptions. The model is required to be a LAN family, which allows to use the convolution theorem of H\'ajek and Le Cam. The results are applied to the forecasting of a bivariate Ornstein-Uhlenbeck process, for which the assumptions of (\cite{bosq2012}) are tricky to verify.
\end{abstract}



\end{frontmatter}


\section{Introduction}

\subsection{Background}

The theory of statistical prediction is an extension of the theory of point estimation where the unknown quantity to infer from the observation is not deterministic but random. It develops by generalizing the results of point estimation to prediction problems (see for instance \cite{yatracos1992}, \cite{bosq-blanke2007} and \cite{bosq2007}). For instance some authors have studied the extension of the Cramér-Rao inequality to the case of statistical prediction problems (\cite{yatracos1992}, \cite{miyata2001}, \cite{nayak2002}, \cite{onzon2011}). 
In the univariate case and for an unbiased predictor $p(X)$ it has the form
\begin{equation}\label{CR_ineq_pred}
\E_\theta (p(X) - r(X,\theta))^2 \geqslant \frac{\big( \E_\theta (\partial_\theta r(X,\theta)) \big)^2}{I(\theta)},
\end{equation}
where $I(\theta)$ is the Fisher information pertaining to the observation $X$, and $r(X,\theta)$ is the conditional expectation of the unobserved variable $Y$ to predict given the observation $X$,
\[
r(X,\theta) = \E_\theta [ Y | X ].
\]
The Cramér-Rao inequality gives a lower bound on the mean squared error of estimation which gives an optimality criterion with which to compare the performance of a given estimator. 
One limitations of this lower bound, however, is the fact that it depends on the bias of the estimator which is applied to. This means that the optimality criterion it provides only allows to compare estimators with the same bias. For this reason, the Cramér-Rao inequality is mainly useful for unbiased estimation.
Another limitation is the fact that it is not necessarily attained by any estimator at all, and an estimator which risk does not attains the bound might still be the optimal one.

The approach that was developped to overcome these limitations, consists in studying the limit of the (normalized) error of estimation when the sample size tends to infinity. This work culminated with the celebrated H\'ajek--Le Cam convolution theorem and the H\'ajek--Le Cam minimax inequality (\cite{hajek1970}, \cite{hajek1972}, \cite{lecam1972}).

The two limitations of the Cramér-Rao inequality discussed above also hold for the extension of the inequality to prediction. Moreover in some instances the lower bound given by the inequality for predictors is zero, which does not give any information at all. This is the case for example in the problem of forecasting the Ornstein-Uhlenbeck process, or more generally, as soon as $\E_\theta (\partial_\theta r(X,\theta)) = 0$. It has been seen in \cite{bosq2012} that the asymptotic approach pursued for estimation could be generalized to prediction problems under some suitable conditions. Besides, the asymptotic lower bound obtained is of the form
\[
\frac{\E_\theta \big( \partial_\theta r(X,\theta) \big)^2}{I(\theta)},
\]
instead of the form of (\ref{CR_ineq_pred}). The square being \emph{inside} the expectation, the lower bound for 
problems with $\E_\theta (\partial_\theta r(X,\theta)) = 0$ is not zero anymore. In the case of the problem of forecasting the univariate stationary Ornstein-Uhlenbeck process, \cite{bosq2012} prove that the asymptotic lower bound is attained by the plug-in predictor with the parameter $\theta$ being estimated by the maximum likelihood estimator (MLE). The results in \cite{bosq2012} are based on a version of the Cramér-Rao bound for the estimation of the regression function $x \mapsto r(x,\theta)$. This lower bound depends on the bias $b(x,\theta)$ of estimation of the regression function and its derivatives $\partial_{\theta_i} b(x,\theta)$. A consequence is that the asymptotic results require the latter to converge towards zero fast enough.
In the case of the forecasting of the univariate Ornstein-Uhlenbeck process, this condition is verified for the plug-in estimator $r(\cdot,\widehat{\theta})$ where $\widehat{\theta}$ is the MLE. This follows from results in \cite{bosq2010} that are derived from the closed form of the MLE $\widehat{\theta}$.
However, depending on the problems, the condition on the speed of convergence of the derivatives of the bias is not always easy to verify.
For instance, in the case of the bivariate Ornstein-Uhlenbeck process, no closed form is known for the MLE. Hence the results of \cite{bosq2010} cannot be generalized directly to the multivariate case. Consequently the verification of the condition on the speed of convergence of the derivatives of the bias is not easy to verify. In this article we present results that rely of a different set of assumptions which are fulfilled in the case of the forecasting of the bivariate Ornstein-Uhlenbeck process.
The asymptotic lower bounds presented here are deduced from the H\'ajek--Le Cam convolution theorem.

In the remainder of this section we set the notations that will be useful throughout the article, pose the problem we are interested in and then give a summary of the main results presented here.

\subsection{Notations}

We will use the following notations. For all $m$ and $n$ positive integers let $\M_{m,n}$ be the space of matrices with $m$ lines and $n$ columns with real coefficients. We will use the Frobenius norm on this space, which we will denote $\|\,\cdot\,\|_{\M_{m,n}}$, and is defined as follows. If $A \in \M_{m,n}$ with coefficients $a_{i,j}$, $1\leqslant i \leqslant m$, $1\leqslant j \leqslant n$ then
\[
\|A\|_{\M_{m,n}} = \sqrt{\sum_{i=1}^m \sum_{j=1}^n a_{i,j}^2} = \sqrt{\trace(A'A)} = \sqrt{ \sum_{i=1}^n \sigma_i^2 }
\]
where $\sigma_i^2$ are the eigenvalues of $A'A$ and where $A'$ is the transpose matrix of $A$. For square matrices, $\M_n = \M_{n,n}$. For a matrix $A\in \M_{m,n}$, we will note
$AA' = A^{\times 2}$.

The symbols $\J_\theta$ and $\nabla_\theta$ represent respectively, the jacobian matrix and gradient operators, with respect to the multidimensional variable $\theta\in\R^d$. The symbol $\partial_\theta$ represents the differentiation with respect to $\theta$ when $\theta$ is a real variable.

When we write an inequality between matrices, it will always refer to L\"owner partial order, \emph{i.e.} $A\leqslant B$ if and only if $A$ and $B$ are square matrices of same size and $B-A$ is a symmetric positive semidefinite matrix.


\subsection{Model and problem}\label{section_modele}

Let $(X_t,\,t \in \mathcal{T})$, where $\mathcal{T} = \N$ or $\mathcal{T} = \R_+$, be a random process taking its values in $E$ where $(E,\mathcal{B})$ is a measurable space. We will note $(X_t,\,t \geqslant 0)$ without refering to $\mathcal{T}$ thereafter. We denote its distribution by $\P_\theta$ where $\theta\in\Theta\subset\R^d$ is an unknown parameter. For all time $T\geqslant 0$, we assume that the path
\[
X_{(T)} = (X_t,\,0\leqslant t \leqslant T),
\]
is observed. Let $(Y_t,\,t \geqslant 0)$ be a random process taking its values in $\R^k$, such that for all $T\geqslant 0$, $Y_T$ is not observed at time $T$. In the application we have in mind, $Y_T = X_{T+h}$ with $h>0$.

In this section we consider the problem of estimating the following regression function,
\[
r(\cdot,\theta) : x \mapsto \E_\theta \left( Y_T | X_T = x \right),
\]
given $X_{(T)}$.
Here the function $r$ is assumed not to depend on $T$, (we did not make this restriction in \cite{bosq2012}).
We will consistently assume that for all $T\geqslant 0$ the function $(x,\theta) \mapsto r(x,\theta)$ is measurable and for all $x\in E$, the function $\theta \mapsto r(x,\theta)$ is differentiable over $\Theta$.

In the application, we have in mind that this regression function is
\[
r(\cdot,\theta) : x \mapsto \E_\theta \left( X_{T+h} | X_T = x \right).
\]

The risk we choose for this problem is the quadratic error of estimation of the regression function (QER),
\[
\rho_T(\theta) = \int_E \E_\theta\big(\widehat{r}_T(x) - r(x,\theta)\big)^{\times 2} d\mu_{\theta}(x),
\]
where we choose $\mu_{\theta}$, a measure over $(E,\mathcal{B})$, with $E$ the range of the random variable $X_T$, and $\widehat{r}_T$ an estimator of the regression function $r(\cdot,\theta)$. Oftentimes $\mu_{\theta}$ will be the distribution of $X_T$ under $\P_\theta$, and $\widehat{r}_T$ will be a plug-in estimator, \emph{i.e.} of the form $\widehat{r}_T = r(\cdot, \widehat{\theta}_T)$. Here the measure $\mu_{\theta}$ is assumed not to depend on $T$ (we did not make this restriction in \cite{bosq2012}).

We also consider the corresponding prediction problem, which consists in predicting
\[
r(X_T,\theta) = \E_\theta \left( Y_T | X_T \right),
\]
given $X_{(T)}$. In the application we have in mind this conditional expectation will be
\[
r(X_T,\theta) = \E_\theta \left( X_{T+h} | X_T \right).
\]
The risk we choose for this problem is the quadratic error of prediction (QEP),
\[
R_T(\theta) = \E_\theta \big( \widehat{r}_T(X_T) - r(X_T,\theta)\big)^{\times 2}.
\]

This prediction problem identifies with prediction of $Y_T$ when the following condition is met,
\begin{equation}\label{equa_prop_type_Markov}
\E_\theta \left( Y_T | X_{(T)} \right) = \E_\theta \left( Y_T | X_T \right).
\end{equation}
For instance, this condition is satisfied for the forecasting problem of a future value $X_{T+h}$ of a Markov process $(X_t,\,t \geqslant 0)$ observed until time $T$, taking $Y_T = X_{T+h}$.

\begin{remarque}
When the relation (\ref{equa_prop_type_Markov}) is not satisfied for the problem at hand, it might however be the case that after a recasting of the problem, it is satisfied. For example, when forecasting an AR($p$) process, the conditional expectation is
\[
\E_\theta \left( X_{T+h} | X_{(T)} \right)
= \E_\theta \left( X_{T+h} | X_T, X_{T-1}, \ldots, X_{T-p+1} \right)
\neq \E_\theta \left( X_{T+h} | X_T \right).
\]
We can however set $Z_t = (X_t, \ldots, X_{t-p+1})$ and then
\[
\E_\theta \left( Z_{T+h} | Z_{(T)} \right)
= \E_\theta \left( Z_{T+h} | Z_T \right).
\]
\end{remarque}

\begin{remarque}
Troughout the paper we will implicitly assume that the following conditions are fulfilled.
For all $x\in E$, $\theta\in \Theta$, $T>0$,
\begin{enumerate}
\item $\J_\theta r(x,\theta)$ exists,
\item $r(\cdot,\theta) \in \L^2(\mu_{\theta})$ and $\J_\theta r(\cdot,\theta) \in \L^2(\mu_{\theta})$. 
\end{enumerate}
\end{remarque}

\subsection{Summary of the results}

In Section~\ref{section_borne_LAN} we establish an asymptotic lower bound for the QER, for all $y\in\R^k$, 
\[
\varliminf_{T\to\infty} T y' \rho_T(\theta) y \geqslant
y' \int_E (\J_\theta r(x,\theta)) I(\theta)^{-1} (\J_\theta r(x,\theta))' d\mu_\theta(x) y,
\]
and deduce a definition of asymptotic efficiency for the corresponding problem of estimation of the regression function.

Under the assumption of asymptotic equivalence between the risks of both problems,
\begin{equation}\label{equa_asympt_eq}
\lim_{T\to\infty} T \rho_T(\theta)
= \lim_{T\to\infty} T R_T(\theta),
\end{equation}
an asymptotic lower bound on the QEP immediately ensues, for all $y\in\R^k$,
\[
\lim_{T\to\infty} T y' R_T(\theta) y \geqslant
y' \int_E (\J_\theta r(x,\theta)) I(\theta)^{-1} (\J_\theta r(x,\theta))' d\mu_\theta(x) y.
\]
From this bound we deduce a definition of asymptotic efficiency for the prediction problem.

In Section~\ref{chapitre_prediction_asymptotique} we give conditions allowing to verify the asymptotic equivalence (\ref{equa_asympt_eq}) in the case of plug-in predictors, \emph{i.e.} of the form $r(X_T,\widehat{\theta}_T)$, where $\widehat{\theta}_T$ is an estimator of $\theta$ based on $X_{(T)}$. Theorems~\ref{risk_equivalence_multi_d} and \ref{prop_risk_eq_bis}, proved in Section~\ref{section_asympt_eq_QEP_QER}, state conditions that imply (\ref{equa_asympt_eq}). The conditions of Theorem~\ref{prop_risk_eq_bis} are fulfilled in our example of forecasting the bivariate Ornstein-Uhlenbeck process.

Finally in Section~\ref{section_decomp_QER} we see that under some conditions the limit of the QER can be written, in the univariate and multivariate cases respectively,
\begin{align*}
\lim_{T\to\infty} T \rho_T(\theta) 
&= \int_E \big( \partial_\theta r(x ,\theta) \big)^2 d\mu_{\theta}(x) \, V(\theta),\\
\lim_{T\to\infty} T \rho_T(\theta) 
&= \int_E \big(\J_\theta r(x ,\theta)\big) V(\theta) \big(\J_\theta r(x ,\theta)\big)' d\mu_{\theta}(x),
\end{align*}
with $V(\theta) =  \lim_{T\to\infty} T\E_\theta(\widehat{\theta}_T-\theta)^{\times 2}$.
The result implies the fulfillment of Assumption~\ref{hypo_rho_T_S} which is required by Theorems~\ref{risk_equivalence_multi_d} and \ref{prop_risk_eq_bis}.

A remarkable consequence of this result is that, under its conditions, the asymptotic efficiency of a plug-in predictor $r(X_T,\widehat{\theta}_T)$ comes down to the asymptotic efficiency of the estimator $\widehat{\theta}_T$.

\section{Asymptotic efficiency for regression and prediction}\label{section_borne_LAN}

We begin by reminding the convolution theorem of H\'ajek and Le Cam and other propositions that will be useful in proving the results of the article.

\subsection{LAN families and H\'ajek--Le Cam convolution theorem}\label{section_eff_asympt_estimation}

In this section we gather definitions and results pertaining to locally asymptotically normal (LAN) families and asymptotic efficiency for reference in the remainder of the article.

References about the theory of asymptotic efficiency are the books of \cite{ibragimov1981}, \cite{lecam1986}, \cite{pfanzagl1994}, \cite{BKRW1998} and \cite{vandervaart1998}.

\begin{definition}\label{definition_LAN}
Let $\Theta\subset\R^d$ be the domain of the parameter $\theta$. A family of probability measures $(\P_\theta^T,\, \theta\in\Theta)$, indexed by $T\in\R_+$ or $T\in\N$, is said \emph{locally asymptotically normal (LAN)} at $\theta\in\Theta$, if there is a family of positive real numbers $(c_T,\,T\geqslant 0)$ such that $c_T \to 0$ as $T\to\infty$, a symmetric positive definite matrix $I(\theta)$ and a family of random vectors $(\Delta_T(\theta),\,T\geqslant 0)$, such that for all $u\in\R^d$,
\[
\log\frac{\d\P_{\theta + c_T u}^T}{\d\P_\theta^T}(X_{(T)})
= u' \Delta_T(\theta) - \frac{1}{2}  u' I(\theta) u + \varepsilon_T(X_{(T)},\theta,u),
\]
and the following convergences hold.
\begin{enumerate}
\item $\P_\theta^T \circ \Delta_T(\theta) \underset{T\to\infty}{\Longrightarrow} \mathcal{N}(0,I(\theta))$
\item $ \P_\theta-\lim_{T\to\infty} \varepsilon_T (X_{(T)}, \theta,u) = 0$.
\end{enumerate}
\end{definition}

In the previous definition we used the following notations. For all probability measures $\P$ and all random variables $X$, $\P \circ X$ represents the distribution of $X$ under $\P$. The notation $\underset{T\to\infty}{\Longrightarrow}$ is used for weak convergence (convergence in distribution).


\begin{remarque}
Numerous families of random processes are LAN. In particular, the property holds for ergodic diffusions (\cite{kutoyants2004}), jump Markov processes (\cite{hopfner1988}, \cite{hopfner1990}), and also for ARMA models (\cite{akritas1982}, \cite{Swensen1985}, \cite{kreiss1987}, \cite{hallin1995}), and many other time series models (for instances \cite{koul1997}, \cite{taniguchi2000}, \cite{drost2008}).
\end{remarque}

\begin{definition}
Assume $(\P_\theta^T,\, \theta\in\Theta)$, $T\geqslant 0$ is a LAN family. Let $\widehat{\psi}_T$, $T\geqslant 0$ be a family of estimators of $\psi(\theta)$ such that for all $T\geqslant 0$, $\widehat{\psi}_T$ is $X_{(T)}$-measurable. We say $\widehat{\psi}_T$, $T\geqslant 0$ is a \emph{family of estimators regular at} $\theta\in\Theta$, if there is a distribution $L(\theta)$ such that for all $u\in\R^d$,
\[
\P_{\theta + c_T u}^T \circ c_T^{-1}\left(\widehat{\psi}_T - \psi(\theta +  c_T u)\right) \underset{T\to\infty}{\Longrightarrow} L(\theta).
\]
\end{definition}

We state the convolution theorem proved independently by H\'ajek and Le Cam (\cite{hajek1970}, \cite{lecam1972}). This result will be useful in Section~\ref{section_borne_LAN}. We give here the statement of \cite{pfanzagl1994}.

\begin{theoreme}\label{theo_convolution}
Let $(\P_\theta^T, \theta\in\Theta)$, $T\geqslant 0$, be a LAN family, with $\Theta\in\R^d$. Let $\psi : \Theta \to \R^k$, $k\leqslant d$, be a differentiable function with Jacobian matrix $\J_\theta \psi (\theta)$ with rank $k$ for all $\theta\in\Theta$. Let $\widehat{\psi}_T$, $T\geqslant 0$, be a family of estimators regular at every $\theta\in\Theta$. Then for all $\theta\in\Theta$ there is a distribution $M(\theta)$ over $(\R^k,\mathcal{B}(\R^k))$ such that
\begin{multline}\label{convergence_produit}
\P_\theta^T \circ \left( \Delta_T(\theta), \, c_T^{-1}(\widehat{\psi}_T - \psi(\theta)) - (\J_\theta \psi(\theta)) I(\theta)^{-1} \Delta_T(\theta) \right) \\
\underset{T\to\infty}{\Longrightarrow} \mathcal{N}(0,I(\theta)) \times M(\theta).
\end{multline}
This implies
\[
L(\theta) = \mathcal{N}(0, (\J_\theta \psi(\theta)) I(\theta)^{-1} (\J_\theta \psi(\theta))') * M(\theta).
\]
\end{theoreme}

The following proposition will allow us to show that a family of estimators is regular. It is a consequence of Proposition~7.1.8 p.227 of \cite{pfanzagl1994}.

\begin{proposition}\label{prop_estim_reg}
Let $\Theta\subset\R^d$ and $(\P_\theta^T, \theta\in\Theta)$, be a LAN family indexed by $T\geqslant 0$. Let $\psi : \Theta \to \R^k$ and $\widehat{\psi}_T$, $T\geqslant 0$, be a family of estimators. Assume there is a family of distributions $(L(\theta), \theta\in\Theta)$ such that the following conditions are fulfilled.
\begin{enumerate}
\item For all $\theta\in\Theta$, there is a neighbourhood $V(\theta)$ of $\theta$ such that the following convergence holds uniformly over $V(\theta)$.
\[
\P_\theta^T \circ  c_T^{-1}\left(\widehat{\psi}_T - \psi(\cdot)\right) \underset{T\to\infty}{\Longrightarrow} L(\theta)
\]
\item $\theta\mapsto L(\theta)$ is continuous for weak convergence, \emph{i.e.}, for all $\theta_0\in\Theta$,
\[
L(\theta) \underset{\theta\to\theta_0}{\Longrightarrow} L(\theta_0).
\]
\end{enumerate}
Then the family of estimators $\widehat{\psi}_T$, $T\geqslant 0$, is regular at every $\theta\in\Theta$.
\end{proposition}

In the previous proposition we have used the notion of uniform weak convergence that is defined as follows.

\begin{definition}
Let $\Theta\subset\R^d$ and $Q_\theta^{(n)}$, for all $n\in\N$, $\theta\in\Theta$, let $Q_\theta^{(n)}$ and $Q_\theta$ be probability measures over $(\R^k,\mathcal{B}(\R^k))$.
The sequence $Q_\theta^{(n)}$, $n\in\N$ converges weakly to $Q_\theta$ uniformly over $\Theta$ if for any bounded continuous function $h$,
\[
\lim_{n\to\infty} \sup_{\theta\in\Theta}
\left| Q_\theta^{(n)}(h) - Q_\theta(h) \right| = 0.
\]
\end{definition}

The following proposition (\cite{pfanzagl1994} p.289) gives an asymptotic expansion for a family of estimators which has the optimal limit in distribution.

\begin{proposition}\label{prop_loi_optimale}
Under conditions of Theorem~\ref{theo_convolution}, any regular family of estimators  $(\widehat{\psi}_T,T\geqslant 0)$ which has the optimal limit in distribution, \emph{i.e.}
\[
L(\theta) = \mathcal{N}(0, (\J_\theta \psi(\theta)) I(\theta)^{-1} (\J_\theta \psi(\theta))'),
\]
has the following asymptotic expansion. As $T\to\infty$,
\[
\sqrt{T} (\widehat{\psi}_T - \psi(\theta)) = (\J_\theta \psi(\theta)) I(\theta)^{-1} \Delta_T(\theta) + o_{\P_\theta^T}(1).
\]
\end{proposition}

\subsection{Asymptotic bounds for LAN families}

Assumptions 2.1 and 2.2 we made in the article (\cite{bosq2012}), in particular about the asymptotic behaviour of the derivative of the bias, may seem somewhat arbitrary and may be difficult to verify in some instances.
In this article we present an alternative approach which allows to dispense with those by using other assumptions.
We begin with the following assumption.

\begin{hypo}\label{hypo_LAN_1}{\ }
\begin{enumerate}
\item\label{condition_LAN_1} The family of models $(\P_\theta^T,\, \theta\in\Theta)$, $T\geqslant 0$, where $\P_\theta^T$ is the distribution of $(X_t,\,0\leqslant t \leqslant T)$ under $\P_\theta$, is LAN at every $\theta\in\Theta$, with $c_T = T^{-1/2}$.
\item\label{condition_suite_est_reg} For all $\theta\in\Theta$, for $\mu_\theta$-almost all $x\in E$, the family of estimators $\widehat{r}_T(x)$ is regular at $\theta$ for estimating $r(x,\theta)$ and the limit in distribution is centered with finite variance, \emph{i.e.} there is a centered distribution $L(x,\theta)$ with finite variance such that the following convergence holds for all $u\in\R^d$
\begin{equation}\label{equa_famille_reguliere}
\P_{\theta + u/\sqrt{T}}^T \circ \sqrt{T}\left(\widehat{r}_T(x) - r\left(x,\theta +  u/\sqrt{T}\right)\right) \underset{T\to\infty}{\Longrightarrow} L(x,\theta).
\end{equation}
\end{enumerate}
\end{hypo}

Condition (\ref{equa_famille_reguliere}), which corresponds to regularity of the family of estimators, makes possible to apply the Theorem~\ref{theo_convolution} (H\'ajek--Le Cam convolution theorem).

\begin{lemme}\label{lemme_variance_asympt_ineg_1d}
Let $Y_n$, $n\in\N$ be a sequence of random variables of $\R$ such that $Y_n \Longrightarrow Y$ where $\E Y=0$ and $\E Y^2 = v$. If $\E Y_n^2 \xrightarrow[n\to \infty]{} w$ then $w \geqslant v$.
\end{lemme}

\begin{proof}
This result follows from Lemma~1.14 p.~437 from \cite{lehmann1998}.
\end{proof}

\begin{proposition}\label{prop_borne_asympt_LAN_1d}
In the univariate case and under Assumption~\ref{hypo_LAN_1}, the following asymptotic lower bound holds for all $\theta\in\Theta$,
\[
\varliminf_{T\to\infty} T \rho_T(\theta) \geqslant
\int_E \frac{(\partial_\theta r(x,\theta))^2} {I(\theta)} d\mu_\theta(x).
\]
\end{proposition}

\begin{proof}
Let $x\in E$ such that condition~\ref{condition_suite_est_reg} of Assumption~\ref{hypo_LAN_1} is fulfilled. According to H\'ajek--Le Cam convolution theorem, for all $\theta\in\Theta$, there is $M_x(\theta)$ such that
\[
L(x,\theta) = \mathcal{N}\left(0,\frac{(\partial_\theta r(x,\theta))^2} {I(\theta)}\right) * M_x(\theta).
\]
Using Lemma~\ref{lemme_variance_asympt_ineg_1d} we deduce
\[
\varliminf_{T\to\infty} T \E_\theta (\widehat{r}_T(x) - r(x,\theta))^2 \geqslant
\frac{(\partial_\theta r(x,\theta))^2} {I(\theta)}.
\]
Now, from Fatou lemma,
\[
\varliminf_{T\to\infty} \int_E T \E_\theta (\widehat{r}_T(x) - r(x,\theta))^2 d\mu_\theta(x)
\geqslant
\int_E \varliminf_{T\to\infty} T \E_\theta (\widehat{r}_T(x) - r(x,\theta))^2 d\mu_\theta(x),
\]
therefore
\[
\varliminf_{T\to\infty} T \rho_T(\theta) \geqslant
\int_E \frac{(\partial_\theta r(x,\theta))^2} {I(\theta)} d\mu_\theta(x).
\]
\end{proof}

\begin{remarque}
It is desirable that the quantities $\varlimsup_{T\to\infty} T \rho_T(\theta)$ and $\varliminf_{T\to\infty} T \rho_T(\theta)$ be as small as possible. Yet, from Proposition~\ref{prop_borne_asympt_LAN_1d}, they are at least equal to
\[
\int_E \frac{(\partial_\theta r(x,\theta))^2} {I(\theta)} d\mu_\theta(x).
\]
Hence Proposition~\ref{prop_borne_asympt_LAN_1d} gives an asymptotic optimality criterion for the estimators of the regression function in the univariate case and in the setting of Assumption~\ref{hypo_LAN_1}. The following definition ensues.
\end{remarque}

\begin{definition}\label{def_eff_asympt_LAN_1d}
We say that a family of estimators $\widehat{r}_T$, $T\geqslant 0$, is \emph{asymptotically efficient} for estimating $r(\cdot,\theta)$ if for all $\theta\in\Theta$
\[
\lim_{T\to\infty} T \rho_T(\theta) =
\int_E \frac{(\partial_\theta r(x,\theta))^2} {I(\theta)} d\mu_\theta(x).
\]
\end{definition}

As a shortcut, we will say that the estimator $\widehat{r}_T$ is asymptotically efficient, to mean that the family of estimators $\widehat{r}_T$, $T\geqslant 0$, is asymptotically efficient.

\begin{hypo}\label{hypo_EQR_EQP_equiv_LAN}
We assume that the following limits exist, are finite, non zero and equal.
\[
\lim_{T\to\infty} T \rho_T(\theta)
= \lim_{T\to\infty} T R_T(\theta).
\]
\end{hypo}

For the application that we have in mind, forecasting of random processes, it is not straightforward to verify this assumption. Theorems~\ref{risk_equivalence_multi_d} and \ref{prop_risk_eq_bis}, derived in Section~\ref{section_asympt_eq_QEP_QER}, give conditions to verify this assumption. Theorem~\ref{prop_risk_eq_bis} will allow us to verify Assumption~\ref{hypo_EQR_EQP_equiv_LAN} for the problem of forecasting of the bivariate stationary Ornstein-Uhlenbeck process in Section~\ref{section_OU_risk_eq}.

The following corollary is deduced immediately.

\begin{corollaire}
In the univariate case, under Assumptions~\ref{hypo_LAN_1} and \ref{hypo_EQR_EQP_equiv_LAN}, for all $\theta\in\Theta$
\[
\lim_{T\to\infty} T R_T(\theta) \geqslant
\int_E \frac{(\partial_\theta r(x,\theta))^2} {I(\theta)} d\mu_\theta(x).
\]
\end{corollaire}

This result provides an asymptotic optimality criterion for the predictors of $r(X_T,\theta)$ in the univariate case and in the setting of Assumptions~\ref{hypo_LAN_1} and \ref{hypo_EQR_EQP_equiv_LAN}. The following definition ensues.

\begin{definition}\label{def_eff_asympt_pred_LAN_1d}
We say that a family of predictors $\widehat{r}(X_T)$, $T\geqslant 0$, is \emph{asymptotically efficient} for predicting $r(X_T,\theta)$ if for all $\theta \in \Theta$,
\[
\lim_{T\to\infty} T R_T(\theta) = \int_E \frac{(\partial_\theta r(x,\theta))^2} {I(\theta)} d\mu_\theta(x).
\]
\end{definition}

We now continue with the generalization of these results to the multivariate case.

\begin{lemme}\label{lemme_variance_asympt_ineg}
Let $Y_n$, $n\in\N$ be a sequence of random vectors of $\R^k$ such that $Y_n \Longrightarrow Y$ where $\E Y=0$ and $\E Y^{\times 2} = V$. let $y\in \R^k$, if $y' (\E Y_n^{\times 2}) y \xrightarrow[n\to \infty]{} w$ then $w \geqslant y' V y$.
\end{lemme}

\begin{proof}
The result follows from Lemma~\ref{lemme_variance_asympt_ineg_1d} applied to the sequence of real random variables $(y' Y_n)$, $n\in\N$.
\end{proof}

\begin{proposition}\label{prop_borne_asympt_LAN_md}
Under Assumption \ref{hypo_LAN_1}, the following asymptotic lower bound holds for all $\theta\in\Theta$ and all $y\in\R^k$
\[
\varliminf_{T\to\infty} T y' \rho_T(\theta) y \geqslant
y' \int_E (\J_\theta r(x,\theta)) I(\theta)^{-1} (\J_\theta r(x,\theta))' d\mu_\theta(x) y.
\]
\end{proposition}

\begin{proof}
The proof is identical to the proof of Proposition~\ref{prop_borne_asympt_LAN_1d} except we apply Lemma~\ref{lemme_variance_asympt_ineg} instead of Lemma~\ref{lemme_variance_asympt_ineg_1d}.
\end{proof}

\begin{remarque}
The smaller
$\varlimsup_{T\to\infty} T y' \rho_T(\theta) y$ and $\varliminf_{T\to\infty} T y' \rho_T(\theta) y$ are, the better the estimator is (asymptotically). Yet, from Proposition~\ref{prop_borne_asympt_LAN_md}, they are at least equal to
\[
y' \int_E (\J_\theta r(x,\theta)) I(\theta)^{-1} (\J_\theta r(x,\theta))' d\mu_\theta(x) y.
\]
Hence Proposition~\ref{prop_borne_asympt_LAN_md} gives an asymptotic optimality criterion for the estimators of the regression function in the multivariate case and in the setting of Assumption~\ref{hypo_LAN_1}. The following definition ensues.
\end{remarque}

\begin{definition}\label{def_eff_asympt_LAN_md}
We say that the family of estimators $\widehat{r}_T$, $T\geqslant 0$, is \emph{asymptotically efficient} for estimating $r(\cdot,\theta)$ if for all $\theta\in\Theta$,
\[
\lim_{T\to\infty} T \rho_T(\theta) =
\int_E (\J_\theta r(x,\theta)) I(\theta)^{-1} (\J_\theta r(x,\theta))' d\mu_\theta(x).
\]
\end{definition}

In the multivariate case as well as in the univariate case, Assumption~\ref{hypo_EQR_EQP_equiv_LAN} allows to obtain immediately an asymptotic bound on QEP.

\begin{corollaire}
In the multivariate case, under Assumptions~\ref{hypo_LAN_1} and \ref{hypo_EQR_EQP_equiv_LAN}, for all $\theta\in\Theta$
\[
\lim_{T\to\infty} T y' R_T(\theta) y \geqslant
y' \int_E (\J_\theta r(x,\theta)) I(\theta)^{-1} (\J_\theta r(x,\theta))' d\mu_\theta(x) y.
\]
\end{corollaire}

The following generalization of Definition~\ref{def_eff_asympt_pred_LAN_1d} to the multivariate case ensues.

\begin{definition}\label{def_eff_asympt_pred_LAN_md}
We say that a family of predictors $\widehat{r}(X_T)$, $T\geqslant 0$, is \emph{asymptotically efficient} for predicting $r(X_T,\theta)$ if for all $\theta \in \Theta$,
\[
\lim_{T\to\infty} T R_T(\theta) = \int_E (\J_\theta r(x,\theta)) I(\theta)^{-1} (\J_\theta r(x,\theta))' d\mu_\theta(x).
\]
\end{definition}

\subsection{Bivariate Ornstein-Uhlenbeck process}\label{section_OU_bivarié}

Here we study a forecasting problem of a bivariate stationary Ornstein-Uhlenbeck process. The regression function is estimated by a plug-in estimator for which the model parameter is estimated by the MLE (maximum likelihood estimator). We are going to see that Assumption~\ref{hypo_LAN_1} is fulfilled and the plug-in estimator is asymptotically efficient.

\subsubsection{Model}

We consider the stationary solution $(X_t \in \R^2,\,t\geqslant 0)$ of the following stochastic differential equation,
\begin{equation}\label{equation_OU_2d}
dX_t = -Q(\theta)X_tdt + dW_t,
\end{equation}
with $(W_t,\,t\geqslant 0)$ a standard bivariate Wiener process (\emph{i.e.} a bivariate random process which components are two independent Wiener processes) and,
\[
Q(\theta) = \alpha I_2 + \beta A =
\begin{pmatrix}
\alpha & \beta \\
\beta & \alpha
\end{pmatrix}
\quad\text{with}\quad
A=\begin{pmatrix}
0 & 1 \\
1 & 0
\end{pmatrix}
\quad\text{and}\quad
\theta =\begin{pmatrix}\alpha \\ \beta\end{pmatrix}.
\]
We fix the following set of parameters,
\[
\Theta = \left\{ \theta\in\R^2 \; | \; Q(\theta) = \alpha I_2 + \beta A \text{ is positive definite} \, \right\}.
\]

Applying Itô formula to $g(t,x) = e^{Q(\theta)t} x$ gives an expression of $X_T$,
\[
X_T = e^{- Q(\theta)T} X_0 + \int_0^T e^{Q(\theta)(t-T)} dW_t.
\]

We derive the variance of $X_T$.
\begin{align*}
\E_\theta X_T^{\times 2} & = e^{-Q(\theta) T} (\E_\theta X_0^{\times 2}) e^{-Q(\theta) T}
+ \E_\theta \left(\int_0^T  e^{-Q(\theta) (t - T)} dW_t \right)^{\times 2}\\
& = e^{-Q(\theta) T} (\E_\theta X_0^{\times 2}) e^{-Q(\theta) T}
+ \int_0^T e^{2 Q(\theta) (t - T)} dt\\
& = e^{-Q(\theta) T} (\E_\theta X_0^{\times 2}) e^{-Q(\theta) T}
+ \frac{Q(\theta)^{-1}}{2} - \frac{Q(\theta)^{-1}}{2} e^{- 2 Q(\theta) T}.
\end{align*}
For a stationary process it follows that
\begin{equation}\label{variance_X_T}
\E_\theta X_T^{\times 2} = \frac{Q(\theta)^{-1}}{2}
= \frac{1}{2(\alpha^2 - \beta^2)} \begin{pmatrix} \alpha & -\beta \\ -\beta & \alpha \end{pmatrix}.
\end{equation}

For any $\theta\in\Theta$ the distribution of $X_{(T)}$ is absolutely continuous with respect to the distribution of the process $\big( U + W_t, \; t\geqslant 0 \big)$ where $U$ is a standard gaussian random vector independent from $(W_t,\, t\geqslant 0)$, let $\nu$ be its distribution.
The density of $X_{(T)}$ with respect to $\nu$ is the function $f_T(\cdot,\theta)\,:\,\mathcal{X}_T\to\R$
\begin{multline*}
f_T(X_{(T)},\theta)  = 2 \sqrt{\det Q(\theta)} \times \\
 \quad \exp\left\{ \alpha T - \frac{1}{2}(X'_0 Q(\theta) X_0 + X'_T Q(\theta) X_T) - \frac{1}{2}\int_0^T \| Q(\theta) X_t \|^2 \, dt \right\},
\end{multline*}
where we used results of chapter III of \cite{jacod1987}, in particular sections~4 and 5.

We now calculate the Fisher information matrix corresponding to $X_{(T)}$. It is the opposite of the expectation of the Hessian matrix of the log-likelihood.
\[
I_T(\theta) = - \E_\theta [ \H_\theta \log (f_T(X_{(T)},\theta)) ].
\]
The log-likelihood is
\begin{align*}
\log (f_T(X_{(T)},\theta)) & = \log 2 + \frac{1}{2}\log(\det Q(\theta))\\
& \quad + \alpha T - \frac{1}{2}(X'_0 Q(\theta) X_0 + X'_T Q(\theta) X_T)\\
& \quad - \frac{1}{2}\int_0^T \| Q(\theta) X_t \|^2 \, dt\\
& = \log 2 + \frac{1}{2}\log(\alpha^2 - \beta^2)\\
& \quad + \alpha T - \frac{1}{2}(X'_0 + X'_T) (\alpha I_2 + \beta A) (X_0 + X_T)\\
& \quad - \frac{1}{2}\int_0^T \| Q(\theta) X_t \|^2 \, dt.
\end{align*}
And
\[
\| Q(\theta) X_t \|^2 = \alpha^2 \|X_t \|^2 + 2\alpha\beta X'_t A X_t + \beta^2 \|X_t \|^2.
\]
We deduce
\[
I_T(\theta) = -\frac{1}{2} \E_\theta \left[\H_\theta \log (\alpha^2 - \beta^2) \right]
+ \frac{T}{2} \E_\theta \left[\H_\theta \| Q(\theta) X_0 \|^2 \right]
= I_0(\theta) + T I(\theta).
\]
With
\[
I(\theta) = \frac{1}{2} \E_\theta \left[\H_\theta \left( \alpha^2 \|X_0 \|^2 + 2\alpha\beta X'_0 A X_0 + \beta^2 \|X_0 \|^2 \right) \right].
\]
Using the expression of the variance of $X_T$ (\ref{variance_X_T}) we deduce,
\[
I(\theta) = Q(\theta)^{-1}.
\]

Let $T>0$ and $h>0$, we consider the problem of predicting $X_{T+h}$ given $X_{(T)}$. Since $(X_t)_{t\geqslant 0}$ is a Markov process, it holds
\[
\E_\theta[X_{T+h}|X_{(T)}] = \E_\theta[X_{T+h}|X_T] =
r(X_T,\theta) = e^{-h Q(\theta)}X_T.
\]
The formula to differentiate a matrix exponential $e^{M(\theta)}$ with respect to $\theta$ is
\[
\frac{\partial\, e^{M(\theta)}}{\partial\theta_i}
= \int_0^1 e^{u M(\theta)} \, \frac{\partial M(\theta)} {\partial\theta_i} \, e^{(1-u)M(\theta)}\,du.
\]
The partial derivatives of $r$ with respect to the components of $\theta$ are
\begin{align*}
\frac{\partial}{\partial\alpha} r(X_T,\theta)
& = -h e^{-h Q(\theta)} X_T,\\
\frac{\partial}{\partial\beta} r(X_T,\theta)
& = -h A e^{-h Q(\theta)} X_T
= -h e^{-h Q(\theta)} A X_T,
\end{align*}
because $A$ and $e^{-h Q(\theta)}$ commute. Let
\[
M_T =
\begin{pmatrix}
X_{1,T} & X_{2,T}\\
X_{2,T} & X_{1,T}
\end{pmatrix}
= X_{1,T} I_2 + X_{2,T} A
\]
with $X_{1,T}$ and $X_{2,T}$ the two components of $X_T$. Then
\[
\J_\theta \left[r(X_T,\theta)\right] = -h e^{-hQ(\theta)} M_T.
\]

\subsubsection{Asymptotic efficiency}\label{section_asympt_eff_OU}

Let
\[
U_T = \J_\theta \left[r(X_T,\theta)\right] = -h e^{-hQ(\theta)} M_T.
\]

To make notations lighter we note $U$, $M$, $X_{(1)}$ and $X_{(2)}$ independent copies of $U_T$, $M_T$, $X_{1,T}$ and $X_{2,T}$ respectively.

Let $\widehat{\theta}_T$ be the MLE of $\theta$. One wants to predict $X_{T+h}$ with the plug-in predictor
\[
\widehat{r}_T(X_T) = r(X_T,\widehat{\theta}_T)
= e^{-hQ(\widehat{\theta}_T)} X_T.
\]
The MLE $\widehat{\theta}_T$ satisfies 
\[
V_T = \sqrt{T}(\widehat{\theta}_T-\theta)\, \underset{T\to\infty}{\Longrightarrow} \,
V \sim \mathcal{N}(0,I(\theta)^{-1}).
\]
It holds
\begin{align*}
\nu_T(\theta)
&= \int_{\R^2} \big( \J_\theta r(z,\theta) \big) \,I_T(\theta)^{-1}\, \big( \J_\theta r(z,\theta) \big)' \, d\mu_\theta(z)\\
&= \E_\theta\left[U I_T(\theta)^{-1} U'\right]\\
&= \E_\theta\left[U I_0(\theta)^{-1} U'\right] + T \nu_*(\theta),
\end{align*}
with
\[
\nu_*(\theta) = \E_\theta\left[U I(\theta)^{-1} U'\right] = \int_E (\J_\theta r(x,\theta)) I(\theta)^{-1} (\J_\theta r(x,\theta))' d\mu_\theta(x).
\]
Then
\[
\left\|\nu_T(\theta) - \frac{\nu_*(\theta)}{T} \right\|_{\M_2} = o\big(\left\|\nu_T(\theta)\right\|_{\M_2}\big).
\]
We calculate $\nu_*(\theta)$
\[
\nu_*(\theta) = h^2 e^{-h Q(\theta)} 
\E_\theta\left[M I(\theta)^{-1} M'\right]
 e^{-h Q(\theta)}, 
\]
and
\begin{align*}
\E_\theta[M I(\theta)^{-1} M']
&= \E_\theta\left[ \big(X_{(1)} I_2 + X_{(2)} A\big)
\big(\alpha I_2 + \beta A\big)
\big(X_{(1)} I_2 + X_{(2)} A\big)\right]\\
&= \E_\theta \Big[
\big(\alpha(X_{(1)}^2+X_{(2)}^2) + 2\beta X_{(1)} X_{(2)}\big) I_2\\
& \quad \quad \quad +\big(\beta(X_{(1)}^2+X_{(2)}^2) + 2 \alpha X_{(1)} X_{(2)}\big) A
\Big]\\
&= \frac{1}{2(\alpha^2-\beta^2)}
\begin{pmatrix}
2\alpha^2-2\beta^2 & 2\alpha\beta - 2\alpha\beta \\
2\alpha\beta - 2\alpha\beta & 2\alpha^2-2\beta^2
\end{pmatrix}
= I_2,
\end{align*}
hence $\E_\theta[M I(\theta)^{-1} M'] = I_2$, therefore
\[
\nu_*(\theta) = h^2 e^{-2h Q(\theta)}.
\]

The QER with respect to $\mu_\theta = \mathcal{N}\big(0,\frac{Q(\theta)^{-1}}{2}\big)$ is
\[
\rho_T(\theta) =
\E_\theta\left[(e^{-h Q(\widehat{\theta}_T)} - e^{-h Q(\theta)})
\E_\theta (X_T X_T')
(e^{-h Q(\widehat{\theta}_T)} -
 e^{-h Q(\theta)})'\right],
\]
hence
\[
\rho_T(\theta) =
\frac{1}{2}
\E_\theta\left[(e^{-h Q(\widehat{\theta}_T)} - e^{-h Q(\theta)})
Q(\theta)^{-1}
(e^{-h Q(\widehat{\theta}_T)} -
 e^{-h Q(\theta)})'\right].
\]
For all $\theta\in\Theta$, it holds
\[
Q(\theta) = P D(\theta) P^{-1}
\quad
\text{and}
\quad
e^{-hQ(\theta)} = P e^{-h D(\theta)} P^{-1}
\]
with
\[
P =
\begin{pmatrix}
1 & 1\\
1 & -1
\end{pmatrix}
\quad
\text{et}
\quad
D(\theta) =
\begin{pmatrix}
\alpha + \beta & 0\\
0 & \alpha - \beta
\end{pmatrix}.
\]
The QER is
\begin{align*}
\rho_T(\theta)
&= \frac{1}{2} \E_\theta\left[
P (e^{-hD(\widehat{\theta}_T)} - e^{-hD(\theta)}) P^{-1} Q(\theta)^{-1} P (e^{-hD(\widehat{\theta}_T)} - e^{-hD(\theta)}) P^{-1}\right]\\
&= \frac{P}{2} \E_\theta\left[
(e^{-hD(\widehat{\theta}_T)} - e^{-hD(\theta)}) D(\theta)^{-1} (e^{-hD(\widehat{\theta}_T)} - e^{-hD(\theta)})\right]P^{-1}\\
&= \frac{P}{2}
\begin{pmatrix}
\frac{\E_\theta\big(e^{-h(\widehat{\alpha}_T + \widehat{\beta}_T)} - e^{-h(\alpha+\beta)}\big)^2} {\alpha+\beta} &
0\\
0 &
\frac{\E_\theta\big(e^{-h(\widehat{\alpha}_T - \widehat{\beta}_T)} - e^{-h(\alpha-\beta)}\big)^2}{\alpha-\beta}
\end{pmatrix}
P^{-1}.
\end{align*}
We consider the reparameterization
\[
\xi=\begin{pmatrix} \eta \\ \gamma \end{pmatrix}
=
\begin{pmatrix} e^{-h(\alpha+\beta)} \\ e^{-h(\alpha-\beta)} \end{pmatrix},
\]
then the MLE of $\xi$ is
\[
\widehat{\xi}_T=\begin{pmatrix} \widehat{\eta}_T \\ \widehat{\gamma}_T \end{pmatrix}
=
\begin{pmatrix} e^{-h(\widehat{\alpha}_T+\widehat{\beta}_T)} \\ e^{-h(\widehat{\alpha}_T-\widehat{\beta}_T)} \end{pmatrix}.
\]
It holds
\[
I(\xi)^{-1} = \big(\J_\theta\xi(\theta)\big) I(\theta)^{-1} \big(\J_\theta\xi(\theta)\big)',
\]
\[
\J_\theta\xi(\theta) = -h
\begin{pmatrix}
e^{-h(\alpha+\beta)} & e^{-h(\alpha+\beta)} \\
e^{-h(\alpha-\beta)} & -e^{-h(\alpha-\beta)}
\end{pmatrix},
\]
hence
\[
I(\xi)^{-1} = 2h^2
\begin{pmatrix}
(\alpha+\beta)e^{-2h(\alpha+\beta)} & 0 \\
0 & (\alpha-\beta)e^{-2h(\alpha-\beta)}
\end{pmatrix}.
\]
The estimator $\widehat{\xi}_T$ satisfies 
\[
T \E_\theta \big(\widehat{\xi}_T-\xi\big) \big(\widehat{\xi}_T-\xi\big)'
\, \xrightarrow[T\to\infty]{} \, I(\xi)^{-1},
\]
hence
\begin{align*}
T \rho_T(\theta) \xrightarrow[T\to\infty]{}
\; & h^2 P
\begin{pmatrix}
e^{-2h(\alpha+\beta)} & 0 \\
0 & e^{-2h(\alpha-\beta)}
\end{pmatrix}
P^{-1}
= h^2 P e^{-2h D(\theta)} P^{-1} \\
& = h^2 e^{-2hQ(\theta)} = \nu_*(\theta).
\end{align*}
Thus $r(\cdot,\widehat{\theta}_T)$ is an asymptotically efficient estimator of $r(\cdot,\theta)$ for the QER with respect to $\mu_\theta = \mathcal{N}\big(0,\frac{Q(\theta)^{-1}}{2}\big)$.

Applying Theorem 2.8 p.121 of \cite{kutoyants2004}, for all $\theta \in \Theta$, there is a compact neighbourhood $V(\theta)$ of $\theta$, such that the following convergence holds uniformly over $V(\theta)$,
\[
\sqrt{T}(e^{-hQ(\widehat{\theta}_T)}x - e^{-hQ(\cdot)}x) \underset{T\to\infty}{\Longrightarrow} \mathcal{N}(0, (\J_\theta (e^{-hQ(\theta)}x)) I(\theta)^{-1} (\J_\theta (e^{-hQ(\theta)}x))'),
\]
where the convergence in distribution is taken with respect to $\P_\theta^T$, the distribution of $X_{(T)}$.
Moreover the model is LAN (see \cite{kutoyants2004}). Hence applying Proposition~\ref{prop_estim_reg}, we deduce that the family of estimators $\left(e^{-h Q( \widehat{\theta}_T) }x, T\geqslant 0 \right)$ is regular. Therefore Assumption~\ref{hypo_LAN_1} is fulfilled.

We conclude that the results pertaining to the estimation of the regression function we have seen in this section apply to the problem of forecasting of the bivariate stationary Ornstein-Uhlenbeck process that we consider here, and the plug-in estimator $r(\cdot,\widehat{\theta}_T)$ is asymptotically efficient. We will see in Section~\ref{chapitre_prediction_asymptotique} the extension of the results to the prediction problem, in particular we will verify Assumption~\ref{hypo_EQR_EQP_equiv_LAN}.

\section{Limits of risks for prediction and regression}\label{chapitre_prediction_asymptotique}

In this section we give conditions under which the QEP and the QER are asymptotically equivalent.
This asymptotic equivalence has been used in the previous section to deduce an asymptotic bound on QEP from an asymptotic bound on QER.
Here we consider the plug-in estimators of the regression function,
$\widehat{r}_T = r(\cdot,\widehat{\theta}_T)$
where $\widehat{\theta}_T$ is an estimator of $\theta$ based on $X_{(T)}$, and we consider the corresponding predictors, of the form $\widehat{r}_T(X_T) = r(X_T,\widehat{\theta}_T)$. The QER will always be taken with respect to $\mu_{\theta} = \P_{\theta, X_T}$ the distribution of $X_T$.

\subsection{Assumptions and lemma}

The QEP of the predictor $r(X_T,\widehat{\theta}_T)$ is
\[
R_T(\theta) = \E_\theta\big(r(X_T,\widehat{\theta}_T) - r(X_T,\theta)\big)^{\times 2},
\quad\theta\in\Theta.
\]
In order to compare $R_T(\theta)$ with $\rho_T(\theta)$, we consider an auxiliary predictor $r(X_T,\widehat{\theta}_{S(T)})$ and the corresponding estimator of the regression function $r(\cdot,\widehat{\theta}_{S(T)})$ where $\widehat{\theta}_{S(T)}$ is based on $X_{(S(T))}$ with $S:\R_+ \to \R_+$ a function such that $S(T) \leqslant T$ for all $T\geqslant 0$ and $S(T) \sim T$, as $T\to\infty$. In what follows, $S$ will always represent $S(T)$, omitting the argument $T$ in order to make notations lighter.

We introduce $R_T^S(\theta)$ the QEP of the predictor $r(X_T,\widehat{\theta}_S)$, and $\rho_T^S(\theta)$ the QER of the estimator $r(\cdot,\widehat{\theta}_S)$. These quantities are
\begin{align*}
R_T^S(\theta) & = \E_\theta \left( r(X_T,\widehat{\theta}_S) - r(X_T,\theta) \right)^{\times 2},\\
\rho_T^S(\theta) & = \int_E \E_\theta \left( r(x,\widehat{\theta}_S) - r(x,\theta) \right)^{\times 2} d\mu_{\theta}(x).
\end{align*}

For all $T\geqslant 0$, we define the function $\bar{\theta}_S : \mathcal{X}_S \to \Theta$, such that $\bar{\theta}_S(X_{(S)}) = \widehat{\theta}_S$, where $\mathcal{X}_S$ is the space of the paths of $X_{(S)}$ equipped with the smallest $\sigma$-field that makes the coordinate applications continuous.
Let
\[
\Delta r (x,\xi) = r(x, \bar{\theta}_S(\xi)) - r(x,\theta),
\quad\forall x\in E,\, \xi\in\mathcal{X}_S,
\]
where the dependence of $\Delta r$ in $T$ and $\theta$ is left implicit, to make notations lighter.

The distribution $\P_{\theta, (X_T,X_{(S)})}$ of $(X_T,X_{(S)})$, is assumed to be dominated by a $\sigma$-finite measure $\lambda$ and $f_{\theta, X_T}$, $f_{\theta, X_{(S)}}$ and $f_{\theta,(X_T,X_{(S)})}$ represent the densities of $X_T$, $X_{(S)}$ and $(X_T,X_{(S)})$ respectively.
We can then write $R_T^S (\theta)$ and $\rho_T^S (\theta)$ in the following way,
\begin{align*}
R_T^S (\theta) 
& = \int_{E\times \mathcal{X}_S} (\Delta r(x,\xi))^{\times 2} d\P_{\theta, (X_T,X_{(S)})}(x,\xi)\\
& = \int_{E\times \mathcal{X}_S} (\Delta r(x,\xi))^{\times 2} f_{\theta,(X_T,X_{(S)})} d\lambda(x,\xi),
\end{align*}
and
\begin{align*}
\rho_T^S (\theta) 
& = \int_{E\times \mathcal{X}_S} (\Delta r(x,\xi))^{\times 2} d\P_{\theta, X_T}(x) d\P_{\theta, X_{(S)}}(\xi) \\
& = \int_{E\times \mathcal{X}_S} (\Delta r(x,\xi))^{\times 2} f_{\theta,X_T}(x)f_{\theta,X_{(S)}}(\xi)d\lambda(x,\xi).
\end{align*}
And let
\[
\Delta f(x,\xi) = | f_{\theta,(X_T,X_{(S)})}(x,\xi) - f_{\theta,X_T}(x)f_{\theta,X_{(S)}}(\xi) |,
\]
letting the dependence of $\Delta f$ in $T$ and $\theta$ implicit.

We measure the dependence of $X_T$ and $X_{(S)}$ with the coefficient
\[
\bar{\beta}(S,T) = \int_{E\times \mathcal{X}_S} (\Delta f)d\lambda.
\]

\begin{remarque}
This coefficient is bounded from above by the usual $\beta$-mixing coefficient. It holds $\bar{\beta} (S,\,T) \leqslant 2 \beta(T-S)$ with
\[
\beta(t) =
\sup_{s>0} \left\| \P_{0,\,\theta}^s\otimes\P_{s+t,\,\theta}^\infty - \P_{s,\,t,\,\theta} \right\|_{\mathrm{TV}}
\]
where $\P_{0,\,\theta}^{s}$ is the distribution of $(X_u, \; 0\leqslant u\leqslant s)$, $\P_{s+t,\,\theta}^\infty$ the distribution of $(X_u, \; u\geqslant s+t)$, $\P_{s,\,t,\,\theta}$ the joint distribution of $\big( (X_u)_{0\leqslant u\leqslant s},$ $(X_u)_{u\geqslant s+t} \big)$, and $\|\cdot\|_{\mathrm{TV}}$ the total variation norm for signed measures, \emph{i.e.} if $\mu$ is a signed measure on a measurable space $\mathcal{A}$, then $\|\mu\|_{\mathrm{TV}} = \sup_{A\in\mathcal{A}}|\mu(A)|$. Intuitively, the coefficient $\bar{\beta}(S,T)$ is a way to quantify the dependence between $X_{(S)}$ and $X_T$. A reference on mixing coefficients is \cite{doukhan1994}. There are other approaches to measure the dependence between the past and the future, see \cite{dedecker2007}.
\end{remarque}

We now make the following assumption.

\begin{hypo}\label{hypo_beta_tilde}{\ }
\begin{enumerate}
\item $\exists m \in (2,\infty]\: : \: \delta_{\theta,m} = \sup_{S,T} \|\Delta r\|_{\L^m((\Delta f)\lambda)} < \infty$,
\item $\bar{\beta}(S,T) = o\left(T^\frac{-m}{m-2}\right)$, as $T\to\infty$.
\end{enumerate}
\end{hypo}

\begin{lemme}\label{lemme_beta}
Under Assumption~\ref{hypo_beta_tilde}, if the limit $\lim_{T\to\infty} T \rho_T^S$ exists and is non zero then,
\[
\lim_{T\to\infty} T \rho_T^S = \lim_{T\to\infty} T R_T^S.
\]
\end{lemme}

The proof is similar to the proof of Lemma~3.1 in \cite{bosq2012}.

We now state a lemma which will be useful for the proof of Theorem~\ref{risk_equivalence_multi_d}.

\begin{lemme}\label{norm_ineq}
Let $U$, $V$ and $W$ be column vectors in $\R^k$, then
\[
\|(U-V)^{\times 2} - (W-V)^{\times 2}\|_{\M_k}
\;\leqslant\;
\|U-W\|_{\R^k}^2 + 2 \|U-W\|_{\R^k} \|W-V\|_{\R^k}
\]
\end{lemme}

\begin{proof}
\begin{align*}
(U-V)^{\times 2} - (W-V)^{\times 2}
& = \big( (U-W) + (W-V) \big)^{\times 2} - (W-V)^{\times 2} \\
& = (U-W)^{\times 2} + (U-W)(W-V)' \\
& \quad + (W-V)(U-W)'.
\end{align*}
Hence
\begin{align*}
\| (U-V)^{\times 2} - (W-V)^{\times 2} \|_{\M_k}
& \leqslant \| (U-W)^{\times 2} \|_{\M_k} + \| (U-W)(W-V)' \|_{\M_k} \\
& \quad \; + \| (W-V)(U-W)' \|_{\M_k} \\
& \leqslant \| U-W \|_{\R^k}^2 + 2 \|U-W\|_{\R^k} \|W-V\|_{\R^k}.
\end{align*}
\end{proof}

\begin{hypo}\label{hypo_theta_T_theta_S}
For all $\theta\in\Theta$, $\lim_{T\to\infty}T \E_\theta \| \widehat{\theta}_T-\widehat{\theta}_S \|_{\R^d}^2 = 0$.
\end{hypo}

\begin{remarque}
Proposition~\ref{prop_thetaST_limit_LAN}, which we will see further, gives conditions under which Assumption~\ref{hypo_theta_T_theta_S} is fulfilled.
\end{remarque}

Finally we make the following assumption.

\begin{hypo}\label{hypo_rho_T_S}
For all $\theta \in \Theta$, the following limits exist and are equal, and let $R(\theta) \in \M_k$ such that,
\[
\lim_{T\to\infty} T \rho_T^S(\theta) = \lim_{T\to\infty} T \rho_T(\theta) = R(\theta).
\]
\end{hypo}

\begin{remarque}\label{rem_hypo_rho_T_S}
It is straightforward to see that Assumption~\ref{hypo_rho_T_S} is fulfilled for our problem of forecasting of the bivariate stationary Ornstein-Uhlenbeck process for $S = T - \sqrt{T}$, replacing $\widehat{\theta}_T$ by $\widehat{\theta}_S$ in the derivation we made in Section~\ref{section_asympt_eff_OU} to calculate the limit of $\rho_T(\theta)$.
\end{remarque}

\subsection{Asymptotic equivalence of QEP and QER}\label{section_asympt_eq_QEP_QER}

\subsubsection{Theorem of asymptotic equivalence of the risks}

The two following theorems give conditions for asymptotic equivalence between QEP and QER. They are a generalization of Propositions~3.1 and 3.2 in \cite{bosq2012} from the univariate to multivariate case.

\begin{theoreme}\label{risk_equivalence_multi_d}
Under Assumptions~\ref{hypo_beta_tilde}, \ref{hypo_theta_T_theta_S} and \ref{hypo_rho_T_S}, if there exists a deterministic constant $C > 0$ telle que
\begin{equation}\label{condition_risk_eq_multi_d}
\|r(X_T,\theta^*) - r(X_T,\theta)\|_{\R^k} \leqslant C \, \|\theta^* - \theta \|_{\R^d},
\quad\forall\theta,\,\theta^*\in\Theta,
\end{equation}
then
$
\lim_{T\to\infty} T R_T = \lim_{T\to\infty} T \rho_T.
$
\end{theoreme}

\begin{proof}
Assumption~\ref{hypo_rho_T_S} entails
\[
\lim_{T\to\infty} T \rho^S_T(\theta) = \lim_{T\to\infty} T \rho_T(\theta) = R(\theta).
\]
Assumption~\ref{hypo_beta_tilde} and Lemma~\ref{lemme_beta} give
$
\lim_{T\to\infty} T R^S_T(\theta) = \lim_{T\to\infty} T \rho^S_T(\theta)
$,
hence
\[
\lim_{T\to\infty} T R^S_T(\theta) = \lim_{T\to\infty} T \rho_T(\theta).
\]
Condition (\ref{condition_risk_eq_multi_d}) implies
\begin{equation*}
T \E_\theta \|r(X_T,\widehat{\theta}_T) - r(X_T,\widehat{\theta}_S)\|_{\R^k}^2 \leqslant C^2 T \E_\theta \| \widehat{\theta}_T - \widehat{\theta}_S \|_{\R^d}^2
\end{equation*}
Combining with Assumption~\ref{hypo_theta_T_theta_S} we obtain
\begin{equation}\label{limit_rTthetaST}
T \E_\theta \|r(X_T,\widehat{\theta}_T) - r(X_T,\widehat{\theta}_S)\|_{\R^k}^2 \xrightarrow[T\to\infty]{} 0
\end{equation}
Let $U=r(X_T,\widehat{\theta}_T)$, $V=r(X_T,\theta)$ and $W=r(X_T,\widehat{\theta}_S)$, then using Lemma~\ref{norm_ineq},
\begin{align*}
\big\|T(R_T - R_T^S)\big\|_{\M_k}
&=  T \big\| \E_\theta \big[ (U-V)^{\times 2} - (W-V)^{\times 2} \big]
\big\|_{\M_k}\\
&\leqslant T \E_\theta \big\| 
(U-V)^{\times 2} - (W-V)^{\times 2}
\big\|_{\M_k}\\
&\leqslant T \E_\theta \left(
\|U-W\|_{\R^k}^2 + 2 \|U-W\|_{\R^k} \|W-V\|_{\R^k} \right)\\
&\leqslant T \E_\theta \|U-W\|_{\R^k}^2\\
& \quad  \quad + 2 \left( T \E_\theta \|U-W\|_{\R^k}^2 \, T \E_\theta \|W-V\|_{\R^k}^2 \right)^{1/2}.
\end{align*}
Using (\ref{limit_rTthetaST}) we get $ T \E_\theta \|U-W\|_{\R^k}^2 \xrightarrow[T\to\infty]{} 0$. And moreover
$$
T \E_\theta \|W-V\|_{\R^k}^2 = \trace\big(T R_T^S(\theta)\big) \xrightarrow[T\to\infty]{} \trace\big(R(\theta)\big).
$$
Therefore
\[
\big\|T(R_T(\theta) - R_T^S(\theta))\big\|_{\M_k}
\xrightarrow[T\to\infty]{} 0.
\]
\end{proof}

Condition (\ref{condition_risk_eq_multi_d}) is somewhat restrictive. For instance it is not fulfilled by our problem of forecasting of an Ornstein-Uhlenbeck process. The following result hinges upon less restrictive conditions.

\begin{theoreme}\label{prop_risk_eq_bis}
We assume the following conditions hold.
\begin{enumerate}
\item $T^2\E_\theta \|\widehat{\theta}_T - \widehat{\theta}_S \|_{\R^d}^4 = \bigo(1)$, as $T\to\infty$.
\item There is a measurable function $\ell : E \to \R_+$ such that for all $\theta, \, \theta^* \in \Theta$,
\begin{equation*}
\|r(X_T,\theta^*) - r(X_T,\theta)\|_{\R^k} \leqslant \ell(X_T) \| \theta^* - \theta \|_{\R^d}.
\end{equation*}
\item $\exists \nu>0$ such that $C_\theta = \sup_T \E_\theta\big(\ell^{4+\nu}(X_T)\big) < \infty$.
\end{enumerate}
Then, under Assumptions~\ref{hypo_beta_tilde}, \ref{hypo_theta_T_theta_S} and \ref{hypo_rho_T_S},
$
\lim_{T\to\infty} T R_T = \lim_{T\to\infty} T \rho_T.
$
\end{theoreme}

We refer to the proof of Proposition~3.2 in \cite{bosq2012}, where the result is proved for $\Theta \subset \R$ and $r$ taking its values in $\R$, the proof makes use of Proposition~3.1 in \cite{bosq2012}, which is the restriction of Theorem~\ref{risk_equivalence_multi_d} to the one-dimensional case. The generalization to the multidimensional case is straightforward, using Theorem~\ref{risk_equivalence_multi_d} instead of Proposition~3.1 in \cite{bosq2012}.

\subsubsection{Verification of Assumption~\ref{hypo_theta_T_theta_S}}

We now give a result that allows to verify Assumption~\ref{hypo_theta_T_theta_S}. We make the following assumption.

\begin{hypo}\label{hypo_prop_thetaST_limit_LAN}
Let $\theta\in\Theta$. Assume the family $(\P_\theta^T, \theta\in\Theta)$, $T\geqslant 0$, satisfies the following conditions.
\begin{enumerate}
\item The family $(\P_\theta^T, \theta\in\Theta)$, $T\geqslant 0$, is LAN.
\item For all $\eta>0$, $\lim_{T\to\infty} \P_\theta^T \left( \| \Delta_T(\theta) - \frac{\sqrt{T}}{\sqrt{S}} \Delta_S(\theta) \|_{\R^d}^2 > \eta \right) = 0$.
\end{enumerate}
The family of estimators $(\widehat{\theta}_T, T\geqslant 0)$, satisfies the following conditions.
\begin{enumerate}\setcounter{enumi}{2}
\item The family $(\widehat{\theta}_T, T\geqslant 0)$, is regular at $\theta$.
\item $\P_\theta^T \circ \sqrt{T} (\widehat{\theta}_T - \theta) \underset{T\to\infty}{\Longrightarrow} \mathcal{N}(0,I(\theta)^{-1})$
\item $\lim_{T\to\infty} T \E_\theta \| \widehat{\theta}_T - \theta \|_{\R^d}^2 = \trace(I(\theta)^{-1})$
\end{enumerate}
\end{hypo}

\begin{proposition}\label{prop_thetaST_limit_LAN}
Let $\theta\in\Theta$. Under Assumption~\ref{hypo_prop_thetaST_limit_LAN},
\[
\lim_{T\to\infty}T \E_\theta \| \widehat{\theta}_T-\widehat{\theta}_S \|_{\R^d}^2 = 0.
\]
\end{proposition}

\begin{proof}
Let $\widetilde{\theta}_T = \frac{1}{2}(\widehat{\theta}_T + \widehat{\theta}_S)$. For all vectors $x$ and $y$, the parallelogram identity is
$$
\|x+y\|^2 + \|x-y\|^2 = 2 \|x\|^2 + 2 \|y\|^2.
$$
Taking $x=\widehat{\theta}_T-\theta$ and $y=\widehat{\theta}_S-\theta$ in the normed vector space $\L^2(\P_\theta)$ we get
$$
\E_\theta \| \widehat{\theta}_T+\widehat{\theta}_S-2\theta \|_{\R^d}^2 + \E_\theta \| \widehat{\theta}_T-\widehat{\theta}_S \|_{\R^d}^2 = 2 \E_\theta \|\widehat{\theta}_T-\theta \|_{\R^d}^2 + 2\E_\theta \|\widehat{\theta}_S-\theta \|_{\R^d}^2.
$$
Hence
\begin{equation*}
4T\E_\theta \| \widetilde{\theta}_T-\theta \|_{\R^d}^2 + T\E_\theta \|\widehat{\theta}_T-\widehat{\theta}_S \|_{\R^d}^2 = 2 T\E_\theta \|\widehat{\theta}_T-\theta \|_{\R^d}^2 + 2 T\E_\theta \|\widehat{\theta}_S-\theta \|_{\R^d}^2.
\end{equation*}
To complete the proof it remains to prove that
\begin{equation*}
\varliminf_{T\to\infty} T\E_\theta \| \widetilde{\theta}_T-\theta \|_{\R^d}^2
\geqslant \trace(I(\theta)^{-1}).
\end{equation*}
From Proposition~\ref{prop_loi_optimale}, the normalized error of estimators $\widehat{\theta}_T$ and $\widehat{\theta}_S$ have the following asymptotic expansions.
\[
\sqrt{T}(\widehat{\theta}_T - \theta) = I(\theta)^{-1} \Delta_T + o_{\P_\theta^T}(1),
\]
\[
\sqrt{S}(\widehat{\theta}_S - \theta) = I(\theta)^{-1} \Delta_S + o_{\P_\theta^S}(1).
\]
Where we omit the argument $\theta$ of $\Delta_T$ and $\Delta_S$ to make notations lighter. We deduce,
\begin{align*}
\sqrt{T}(\widetilde{\theta}_T - \theta)
& = \frac{1}{2} \sqrt{T}(\widehat{\theta}_T - \theta)
+ \frac{1}{2} \frac{\sqrt{T}}{\sqrt{S}}\sqrt{S}(\widehat{\theta}_S - \theta)\\
& = \frac{1}{2} I(\theta)^{-1} \left( \Delta_T + \textstyle{ \frac{\sqrt{T}}{\sqrt{S}} } \Delta_S \right) + o_{\P_\theta^T}(1)\\
& =  I(\theta)^{-1} \left( \Delta_T + \frac{1}{2} \left( \textstyle{ \frac{\sqrt{T}}{\sqrt{S}} } \Delta_S - \Delta_T \right) \right) + o_{\P_\theta^T}(1).
\end{align*}
Yet $\P_\theta^T \circ \Delta_T \Rightarrow \mathcal{N}(0,I(\theta))$ and $\frac{\sqrt{T}}{\sqrt{S}} \Delta_S - \Delta_T$ converges toward $0$ in probability. Hence, applying Slutzki's lemma we get,
\[
\P_\theta^T \circ \sqrt{T}(\widetilde{\theta}_T - \theta)
\underset{T\to\infty}{\Longrightarrow} \mathcal{N}(0,I(\theta)^{-1}).
\]
Therefore, from Lemma~\ref{lemme_variance_asympt_ineg},
\[
\varliminf_{T\to\infty} T \E_\theta \| \widetilde{\theta}_T - \theta \|_{\R^d}^2 \geqslant \trace(I(\theta)^{-1})).
\]
This completes the proof.
\end{proof}

\subsubsection{Bivariate Ornstein-Uhlenbeck process}\label{section_OU_risk_eq}

We have seen in Remark~\ref{rem_hypo_rho_T_S} that Assumption~\ref{hypo_rho_T_S} is fulfilled for our problem of forecasting of the bivariate stationary Ornstein-Uhlenbeck process.
We are now going to see that Assumptions~\ref{hypo_beta_tilde} and \ref{hypo_theta_T_theta_S}, and conditions of Theorem~\ref{prop_risk_eq_bis} are fulfilled too.

We begin with Assumption~\ref{hypo_beta_tilde}.
We take $S = T - \sqrt{T}$, and $\lambda = \ell \otimes \nu$ with $\ell$ Lebesgue's measure over $\R$ and $\nu$ the distribution of the process $(U + W_t,\; t\geqslant 0)$, where $(W_t, t\geqslant 0)$ is a standard bivariate Wiener process and $U \sim \mathcal{N}(0,I_2)$ independent (with $I_2$ the unit matrix of $\M_2$). It holds
\[
\|\Delta r (x,\xi)\|_{\R^2} = \|(e^{-Q(\bar{\theta}_T(\xi))h} - e^{-Q(\theta) h})x\|_{\R^2} \leqslant 2 \| x \|_{\R^2}, \quad \forall x\in \R^2.
\]
We deduce that the first condition of Assumption~\ref{hypo_beta_tilde} is fulfilled,
\[
\|\Delta r\|_{\L^m(\Delta f \lambda)}^m \leqslant \int \|\Delta r\|_m^m \Delta f d\lambda
\leqslant 4 \E_\theta \|X_0\|_m^m < \infty,
\]
where $\|\cdot\|_m$ is the $m$ norm of $\R^2$.
In particular for $m=4$, this bound is
\[
4 \E_\theta \|X_0\|_4^4 = 12 ((\E_\theta X_{0,1}^2)^2 + (\E_\theta X_{0,2}^2)^2) = \frac{12 \alpha^2}{(\alpha^2-\beta^2)^2}.
\]

The multivariate Ornstein-Uhlenbeck process is geometrically $\beta$-mixing (see \cite{veretennikov1987}), which allows to conclude that the second condition of Assumption~\ref{hypo_beta_tilde} is fulfilled for all $m>2$.

We are now going to see that Assumption~\ref{hypo_prop_thetaST_limit_LAN} is verified, which will imply Assumption~\ref{hypo_theta_T_theta_S}. We have already seen that conditions~1, 3, 4 et 5 are fulfilled, we turn to condition~2.

The family $\left(\P_\theta^T,\theta\in\Theta\right)$, $T>0$, is LAN (see \cite{kutoyants2004} p.113) with,
\[
\Delta_T(\theta) = T^{-1/2} \int_0^T \left( \J_\theta S(\theta,X_t) \right)' dW_t,
\]
where $S(\theta,X_t) = -Q(\theta)X_t$. Therefore $\Delta_T(\theta) = T^{-1/2} \int_0^T Q(X_t) dW_t$. Now
\begin{align*}
\Delta_T - \frac{\sqrt{T}}{\sqrt{S}} \Delta_S
& = - T^{-1/2} \int_0^T Q(X_t) dW_t + T^{1/2}S^{-1} \int_0^S Q(X_t) dW_t\\
& = T^{-1/2} \left( \frac{T}{S} - 1 \right) \int_0^S Q(X_t) dW_t + T^{-1/2} \int_S^T Q(X_t) dW_t\\
& = A_t + B_t.
\end{align*}
\begin{align*}
\E_\theta \| A_t \|_{\R^2}^2 & = T^{-1} \left( \frac{T}{S} - 1 \right)^2 \int_0^S \E_\theta Q(X_t)^2 dt\\
& = \frac{S}{T} \left( \frac{T}{S} - 1 \right)^2 Q(\theta)^{-1} \xrightarrow[T\to\infty]{} 0.
\end{align*}
The convergence of $A_t$ to $0$ in $\L^2$ implies its convergence in probability.
\begin{align*}
\E_\theta \| B_t \|_{\R^2}^2 & = T^{-1} \int_S^T \E_\theta Q(X_t)^2 dt = \frac{T-S}{T} Q(\theta)^{-1} \\
&  = \left( 1 - \frac{S}{T} \right) Q(\theta)^{-1} \xrightarrow[T\to\infty]{} 0.
\end{align*}
Hence $B_t$ converges to $0$ in probability. We deduce
\[
\left\| \Delta_T - \textstyle{\frac{\sqrt{T}}{\sqrt{S}}} \Delta_S \right\|_{\R^2} \leqslant \| A_t \|_{\R^2} + \| B_t \|_{\R^2} \xrightarrow[T\to\infty]{\P_\theta} 0.
\]
Therefore condition~2 of Assumption~\ref{hypo_prop_thetaST_limit_LAN} is fulfilled.

Condition~1 of Theorem~\ref{prop_risk_eq_bis} is fulfilled from Theorem~2.8 p.121 of \cite{kutoyants2004}. We are now going to see that condition~2 is fulfilled too.
\begin{align*}
\| r(x,\theta) - r(x,\theta^*) \|_{\R^2}
& = \| \J_\theta r(x,\tilde{\theta}) (\theta - \theta^*) \|_{\R^2}\\
& \leqslant \| \J_\theta r(x,\tilde{\theta}) \|_{\M_2} \| \theta - \theta^* \|_{\R^2}\\
& \leqslant h \| e^{-h Q(\tilde{\theta})} \|_{\M_2} \|M(x)\|_{\M_2} \| \theta - \theta^* \|_{\R^2}\\
& \leqslant \frac{h}{2} \| P \|_{\M_2}^2 \| e^{-h D(\tilde{\theta})} \|_{\M_2} \|M(x)\|_{\M_2} \| \theta - \theta^* \|_{\R^2}\\
& \leqslant \sqrt{2} h \| P \|_{\M_2}^2 \| e^{-h D(\tilde{\theta})} \|_{\M_2} \|x \|_{\R^2} \| \theta - \theta^* \|_{\R^2}\\
& \leqslant \ell(x) \| \theta - \theta^* \|_{\R^2},
\end{align*}
with $\tilde{\theta} = \lambda \theta + (1-\lambda) \theta^*$, $\lambda \in [0,1]$, and $\ell(x) = \sqrt{2} h \| P \|_{\M_2}^2 \| e^{-h D(\tilde{\theta})} \|_{\M_2} \|x \|_{\R^2}$. For all $T\geqslant 0$, $X_T \sim \mathcal{N} \left( 0,\frac{Q(\theta)^{-1}}{2} \right)$ hence $\ell$ verifies condition~3 of Theorem~\ref{prop_risk_eq_bis}, for all $\nu>0$.

In conclusion, Assumption~\ref{hypo_EQR_EQP_equiv_LAN} is verified. The results pertaining to prediction of the previous section applies to the problem of forecasting of the bivariate stationary Ornstein-Uhlenbeck process that we consider here. The plug-in predictor $r(X_T,\widehat{\theta}_T)$ is asymptotically efficient.

\subsection{Break down of the limiting QER}\label{section_decomp_QER}

In this section we see results which give an expression of the limiting QER under some conditions. As a by-product we obtain the fulfillment of Assumption~\ref{hypo_rho_T_S}. We shall use the following notation, for all function $f \in \L^2(\mu_\theta)$,
\[
\| f \|_{\mu_{\theta}}^2 = \int_E \|f(x)\|^2 d\mu_\theta(x).
\]

In the univariate case, the limiting QER may take the following product form,
\[
\lim_{T\to\infty} T \rho_T(\theta) = \|\partial_\theta r(\cdot,\theta)\|_{\mu_{\theta}}^2 V(\theta),
\]
with $V(\theta) = \lim_{T\to\infty} T\E_\theta(\widehat{\theta}_T-\theta)^2$. 
The result generalizes to the multivariate case with the following forms,
\begin{align*}
\lim_{T\to\infty} T \rho_T(\theta) 
&= \lim_{T\to\infty} T \int_E \big(\J_\theta r(x ,\theta)\big) \E_\theta(\widehat{\theta}_T-\theta)^{\times 2} \big(\J_\theta r(x ,\theta)\big)' d\mu_{\theta}(x),\\
\lim_{T\to\infty} T \rho_T(\theta) 
&= \int_E \big(\J_\theta r(x ,\theta)\big) V(\theta) \big(\J_\theta r(x ,\theta)\big)' d\mu_{\theta}(x),
\end{align*}
with $V(\theta) = \lim_{T\to\infty} T\E_\theta(\widehat{\theta}_T-\theta)^{\times 2}$. 
As a corollary we will deduce that Assumption~\ref{hypo_rho_T_S} is verified, namely
\[
\lim_{T\to\infty} T \rho_T(\theta) = \lim_{T\to\infty} T \rho_T^S(\theta).
\]

\subsubsection{Univariate case}

\begin{hypo}\label{hypo_EQR_limit}{\ }
\begin{enumerate}
\item\label{condition_c_T} $\exists \alpha\in(0,1]$, $\exists c(x)$ such that $\|c\|_{\mu_{\theta}} < \infty$ and
\[
|\partial_\theta r(x,\theta') - \partial_\theta r(x,\theta)| \leqslant c(x)|\theta'-\theta|^\alpha,
\; \forall x\in E,\;
\forall \theta,\,\theta'\in\Theta,
\; \forall T>0.
\]
\end{enumerate}
Moreover, $\widehat{\theta}_T$ is an estimator of $\theta$ such that
\begin{enumerate}\setcounter{enumi}{2}
\item $\exists V(\theta)$ such that $\lim_{T\to\infty} T\E_\theta(\widehat{\theta}_T-\theta)^2 = V(\theta)$,
\item\label{condition_estim_conv_rate} $T\E_\theta|\widehat{\theta}_T - \theta|^{2+2\alpha} = o(1)$.
\end{enumerate}
\end{hypo}

\begin{proposition}\label{EQR_limit}
Under Assumption~\ref{hypo_EQR_limit}, for all $\theta\in\Theta$,
\[
\lim_{T\to\infty} T\rho_T(\theta) = \|\partial_\theta r(\cdot,\theta)\|_{\mu_{\theta}}^2 V(\theta).
\]
\end{proposition}

\begin{proof}
\[
r(x,\widehat{\theta}_T)-r(x,\theta) = (\widehat{\theta}_T-\theta)\partial_\theta r(x,\widetilde{\theta}_T)
\]
where $\widetilde{\theta}_T \in \big[ \min(\theta,\widehat{\theta}_T), \,\max(\theta,\widehat{\theta}_T) \big]$. Soit
\[
\delta_T(x) = (\widehat{\theta}_T- \theta)(\partial_\theta r(x,\widetilde{\theta}_T) - \partial_\theta r(x,\theta)).
\]
Hence
\begin{equation*}
r(x,\widehat{\theta}_T)-r(x,\theta) = (\widehat{\theta}_T-\theta)\partial_\theta r(x,\theta) + \delta_T(x),
\end{equation*}
and using condition~\ref{condition_c_T} of Assumption~\ref{hypo_EQR_limit}
\begin{equation}\label{delta_maj}
|\delta_T(x)| \leqslant c(x)|\widehat{\theta}_T-\theta|^{1+\alpha}.
\end{equation}
Besides,
\begin{align*}
\rho_T(\theta)
& = \E_\theta(\widehat{\theta}_T-\theta)^2\int_E (\partial_\theta r(x,\theta))^2d\mu_{\theta}(x)
+\int_E\E_\theta(\delta_T^2(x))d\mu_{\theta}(x)\\
& \quad\; + 2\int_E \E_\theta\big[(\widehat{\theta}_T - \theta)\delta_T(x)\big]\partial_\theta r(x,\theta) d\mu_{\theta}(x)\\
& = J_1 + J_2 + J_3.
\end{align*}
From condition (\ref{delta_maj}) it ensues
\[
|J_2| \leqslant \E_\theta(|\widehat{\theta}_T-\theta|^{2+2\alpha})\int_E c^2_T d\mu_{\theta}.
\]
Since $\|c\|_{\mu_{\theta}}$ is bounded, condition~\ref{condition_estim_conv_rate} of Assumption~\ref{hypo_EQR_limit} implies $|J_2|=o(\frac{1}{T})$. Now (\ref{delta_maj}) gives
\begin{align*}
|J_3|
& \leqslant 2 \E_\theta(|\widehat{\theta}_T-\theta|^{2+\alpha})\int_E c(x) \, | \partial_\theta r(x,\theta) | \, d\mu_{\theta}(x)\\
& \leqslant 2 \E_\theta|\widehat{\theta}_T-\theta|^{2+\alpha}\|c\|_{\mu_{\theta}} \|\partial_\theta r(\cdot,\theta)\|_{\mu_{\theta}}\\
& \leqslant 2 \E_\theta\left[ |\widehat{\theta}_T-\theta||\widehat{\theta}_T-\theta|^{1+\alpha} \right] \|c\|_{\mu_{\theta}} \|\partial_\theta r(\cdot,\theta)\|_{\mu_{\theta}}\\
& \leqslant 2 \left(\E_\theta (\widehat{\theta}_T-\theta)^2\right)^{\frac{1}{2}} \left(\E_\theta |\widehat{\theta}_T-\theta|^{2+2\alpha}\right)^{\frac{1}{2}} \|c\|_{\mu_{\theta}} \|\partial_\theta r(\cdot,\theta)\|_{\mu_{\theta}}\\
& \leqslant \bigo\left(\frac{1}{\sqrt{T}}\right) o\left(\frac{1}{\sqrt{T}}\right)
= o\left(\frac{1}{T}\right).
\end{align*}
Finally $T J_1 \longrightarrow \|\partial_\theta r(\cdot,\theta)\|_{\mu_{\theta}}^2 V(\theta)$.
\end{proof}

We deduce that under conditions of Proposition~\ref{EQR_limit}, Assumption~\ref{hypo_rho_T_S} is verified.

\begin{corollaire}
under Assumption~\ref{hypo_EQR_limit}, for all $\theta\in\Theta$,
\[
\lim_{T\to\infty} T \rho_T(\theta) = \lim_{T\to\infty} T \rho_T^S(\theta) = \|\partial_\theta r(\cdot,\theta)\|_{\mu_{\theta}}^2 V(\theta).
\]
\end{corollaire}

\begin{proof}
If $\widehat{\theta}_T$ satisfies Assumption~\ref{hypo_EQR_limit} then $\widehat{\theta}_S$ also satisfies this assumption because $S \sim T$. Hence 
$
\lim_{T\to\infty} T \rho_T^S(\theta) = \|\partial_\theta r(\cdot,\theta)\|_{\mu_{\theta}}^2 V(\theta).
$
\end{proof}

\begin{remarque}
When Proposition~\ref{EQR_limit} applies, the asymptotic efficiency of a plug-in estimator $r(\cdot,\widehat{\theta}_T)$, or a plug-in predictor $r(X_T,\widehat{\theta}_T)$, comes down to the asymptotic efficiency of the estimator $\widehat{\theta}_T$.
More precisely, if Assumption~\ref{hypo_EQR_limit} holds and if $V(\theta) \neq 0$, then $r(\cdot,\widehat{\theta}_T)$ is asymptotically efficient iff
\[
\lim_{T\to\infty} T \E_\theta (\widehat{\theta}_T - \theta)^2 = I(\theta)^{-1},
\]
assuming the family is LAN and $I(\theta)$ is the asymptotic Fisher information.

If in addition Assumptions~\ref{hypo_beta_tilde} and \ref{hypo_theta_T_theta_S} are satisfied, then
\[
\lim_{T\to\infty} T R_T(\theta) = \|\partial_\theta r(\cdot,\theta)\|_{\mu_{\theta}}^2 V(\theta),
\]
and a plug-in predictor $r(X_T,\widehat{\theta}_T)$ is asymptotically efficient iff
\[
\lim_{T\to\infty} T \E_\theta (\widehat{\theta}_T - \theta)^2 = I(\theta)^{-1}.
\]
\end{remarque}

\subsubsection{Multivariate case}

We now see generalizations of these results to the multivariate case.

\begin{hypo}\label{hypo_EQR_limit_multi_d}{\ }
\begin{enumerate}
\item\label{condition_c_T_multi_d}  $\exists\alpha\in(0,1]$, $\exists c:\R^k\to\R$ with $\|c\|_{\mu_{\theta}} < \infty$ such that
\[
\| \J_\theta r(x,\theta ') - \J_\theta r(x,\theta) \|_{\M_{k,d}} \leqslant c(x)\| \theta '-\theta \|_{\R^d}^\alpha,
\quad
x\in E,\;\theta,\,\theta ' \in\Theta.
\]
\end{enumerate}
Moreover, $\widehat{\theta}_T$ is an estimator of $\theta$ such that
\begin{enumerate}\setcounter{enumi}{1}
\item $T\E_\theta \|\widehat{\theta}_T-\theta\|_{\R^d}^2 = \bigo(1)$.
\item\label{condition_estim_conv_rate_multi_d} $T\E_\theta\|\widehat{\theta}_T-\theta\|_{\R^d}^{2+2\alpha} = o(1)$.
\item\label{condition_eqr_limit_multi_d} The following limit exists,
\[
R(\theta) =
\lim_{T\to\infty} T \int_E \big(\J_\theta r(x ,\theta)\big) \E_\theta(\widehat{\theta}_T-\theta)^{\times 2} \big(\J_\theta r(x ,\theta)\big)' d\mu_{\theta}(x).
\]
\end{enumerate}
\end{hypo}

\begin{proposition}\label{EQR_limit_multi_d}
Under Assumption~\ref{hypo_EQR_limit_multi_d}, for all $\theta\in\Theta$,
\[
\lim_{T\to\infty} T\rho_T(\theta) = R(\theta).
\]
\end{proposition}

\begin{proof}
\[
r(x,\widehat{\theta}_T)-r(x,\theta)
= \big( \J_\theta r(x,\widetilde{\theta}_T) \big) (\widehat{\theta}_T-\theta),
\]
where $\widetilde{\theta}_T = \lambda \theta + (1-\lambda) \widehat{\theta}_T$ for some $\lambda\in[0,1]$. Let
\[
\delta_T(x,\,\theta)
= \big(\J_\theta r(x,\widetilde{\theta}_T) - \J_\theta r(x,\theta)\big)(\widehat{\theta}_T-\theta).
\]
Then
\[
r(x,\widehat{\theta}_T)-r(x,\theta)
= \big( \J_\theta r(x,\theta) \big) (\widehat{\theta}_T-\theta) + \delta_T(x,\theta).
\]
Thus
\[
\rho_T(\theta) = J_1 + J_2 + J_3,
\]
with
\begin{align*}
J_1 &= \int_E \big( \J_\theta r(x,\theta) \big) \E_\theta \big(\widehat{\theta}_T-\theta\big)^{\times 2}
\big( \J_\theta r(x,\theta) \big)'
\, d\mu_{\theta}(x),\\
J_2 &= \int_E \E_\theta \big(\delta_T(x,\theta)\delta_T'(x,\theta)\big)
\, d\mu_{\theta}(x),\\
J_3 &=
\int_E \E_\theta \big[ \big( \J_\theta r(x,\theta) \big) \big(\widehat{\theta}_T-\theta\big) \delta_T'(x,\theta) \\
& \quad \quad \quad \quad +
\delta_T(x,\theta) \big(\widehat{\theta}_T-\theta\big)'
\big( \J_\theta r(x,\theta) \big)' \big]
\, d\mu_{\theta}(x).
\end{align*}
Condition~\ref{condition_c_T_multi_d} of Assumption~\ref{hypo_EQR_limit_multi_d} implies
\begin{equation}\label{ineq_delta}
\|\delta_T(x,\theta)\|_{\R^k} \leqslant c(x) \|\widehat{\theta}_T-\theta\|_{\R^d}^{1+\alpha}.
\end{equation}
Hence
\[
\|J_2\|_{\M_k} \leqslant \E_\theta \big( \|\widehat{\theta}_T-\theta\|_{\R^d}^{2+2\alpha} \big) 
\|c\|^2_{\mu_{\theta}}.
\]
Then condition~\ref{condition_estim_conv_rate_multi_d} of Assumption~\ref{hypo_EQR_limit_multi_d} implies
\[
\|J_2\|_{\M_k} = o\Big(\frac{1}{T}\Big).
\]
Now,
\begin{align*}
\|J_3\|_{\M_k} \leqslant 2 \int
\big\| \J_\theta r(x,\theta) \big\|_{\M_{k,d}} \E_\theta \big[ \|\widehat{\theta}_T-\theta \|_{\R^d} \, \| \delta_T(x,\theta) \|_{\R^k} \big]
\, d\mu_{\theta}(x).
\end{align*}
Using (\ref{ineq_delta}) it ensues,
\begin{align*}
\|J_3\|_{\M_k} & \leqslant \, 2 \E_\theta \|\widehat{\theta}_T-\theta\|_{\R^d}^{2+\alpha}
\int_E |c(x)| \big\| \J_\theta r(x,\theta) \big\|_{\M_{k,d}} \, d\mu_{\theta}(x)\\
& \leqslant \, 2 \E_\theta \|\widehat{\theta}_T-\theta\|_{\R^d}^{2+\alpha}
\|c\|_{\mu_{\theta}}
\big\| \J_\theta r(\,\cdot\,,\theta) \big\|_{\mu_{\theta}}\\
& \leqslant \, 2 \E_\theta \left[ \|\widehat{\theta}_T-\theta\|_{\R^d} \|\widehat{\theta}_T-\theta\|_{\R^d}^{1+\alpha} \right]
\|c\|_{\mu_{\theta}}
\big\| \J_\theta r(\,\cdot\,,\theta) \big\|_{\mu_{\theta}}\\
& \leqslant \, 2 \left( \E_\theta \|\widehat{\theta}_T-\theta\|_{\R^d}^2 \right)^{\frac{1}{2}} \left( \E_\theta \|\widehat{\theta}_T-\theta\|_{\R^d}^{2+2\alpha} \right)^{\frac{1}{2}}
\|c\|_{\mu_{\theta}}
\big\| \J_\theta r(\,\cdot\,,\theta) \big\|_{\mu_{\theta}}\\
& \leqslant \bigo\left(\frac{1}{\sqrt{T}}\right) o\left(\frac{1}{\sqrt{T}}\right)
= o\left(\frac{1}{T}\right).
\end{align*}
Finally, from condition~\ref{condition_eqr_limit_multi_d} of Assumption~\ref{hypo_EQR_limit_multi_d}, 
$
T J_1 \underset{T\to \infty}{\longrightarrow} R(\theta).
$
\end{proof}

We deduce that under the conditions of Proposition~\ref{EQR_limit_multi_d}, Assumption~\ref{hypo_rho_T_S} is satisfied.

\begin{corollaire}
Under Assumption~\ref{hypo_EQR_limit_multi_d}, for all $\theta\in\Theta$,
\[
\lim_{T\to\infty} T \rho_T(\theta) = \lim_{T\to\infty} T \rho_T^S(\theta) = R(\theta).
\]
\end{corollaire}

\begin{proof}
The proof is similar to the univariate case. If $\widehat{\theta}_T$ satisfies Assumption~\ref{hypo_EQR_limit_multi_d} then $\widehat{\theta}_S$ also satifies this assumption because $S \sim T$. Hence
\[
\lim_{T\to\infty} T \rho_T^S(\theta) = R(\theta).
\]
\end{proof}

The asymptotic efficiency of a plug-in estimator or a plug-in predictor may comes down to the asymptotic efficiency of the estimator of the parameter, as in the univariate case. For this we have to make the following additional assumption.

\begin{hypo}\label{hypo_EQR_limit_multi_d_2}{\ }
\begin{enumerate}
\item\label{condition_c_T_multi_d_2}  $\exists\alpha\in(0,1]$, $\exists c:\R^k\to\R$ with $\|c\|_{\mu_{\theta}} < \infty$ such that
\[
\|\J_\theta r(x,\theta ')-\J_\theta r(x,\theta)\|_{\M_{k,d}} \leqslant c(x)\|\theta '-\theta\|_{\R^d}^\alpha,
\quad
x\in E,\;\theta,\,\theta ' \in\Theta.
\]
\item For all $\theta\in\Theta$, there is a measurable function $g_\theta : E\to\R_+$ such that $\int_E g_\theta(x) d\mu_\theta(x) < \infty$ and, for $\mu_\theta$-almost all $x$, for all $T\geqslant 0$,
\[
\| \J_\theta r(x,\theta) \|_{\M_{k,d}}^2 \leqslant g_\theta(x).
\]
\end{enumerate}
Moreover, $\widehat{\theta}_T$ is an estimator of $\theta$ such that, for all $\theta\in\Theta$,
\begin{enumerate}\setcounter{enumi}{2}
\item\label{condition_estim_limit_multi_d_2} $\exists V(\theta) \in\M_d$, such that $\lim_{T\to\infty} T\E_\theta (\widehat{\theta}_T - \theta)^{\times 2} = V(\theta)$.
\item\label{condition_estim_conv_rate_multi_d_2} $T\E_\theta\|\widehat{\theta}_T-\theta\|_{\R^d}^{2+2\alpha} = o(1)$.
\end{enumerate}
\end{hypo}

When Assumption~\ref{hypo_EQR_limit_multi_d_2} is satified, we let
\[
R(\theta) = \int_E \J_\theta r(x ,\theta) V(\theta) \big(\J_\theta r(x ,\theta)\big)' d\mu_\theta(x).
\]

\begin{proposition}\label{EQR_limit_multi_d_2}
Under Assumption~\ref{hypo_EQR_limit_multi_d_2}, for all $\theta\in\Theta$,
\[
\lim_{T\to\infty} T \rho_T(\theta) = \int_E \J_\theta r(x ,\theta) V(\theta) \big(\J_\theta r(x ,\theta)\big)'  d\mu_\theta(x)
\]
\end{proposition}

\begin{proof}
We are going to apply Lebesgue's dominated convergence theorem, as $T\to\infty$, to the integral
\begin{align*}
A_T &= T \int_E \big(\J_\theta r(x ,\theta)\big) \E_\theta(\widehat{\theta}_T-\theta)^{\times 2} \big(\J_\theta r(x ,\theta)\big)' d\mu_{\theta}(x).
\end{align*}
Let $y\in\R^k$, for $\mu_\theta$-almost all $x$,
\[
y' T \big(\J_\theta r(x ,\theta)\big) \E_\theta(\widehat{\theta}_T-\theta)^{\times 2} \big(\J_\theta r(x ,\theta)\big)' y
\xrightarrow[T\to\infty]{} \J_\theta r(x,\theta) V(\theta) \big(\J_\theta r(x ,\theta) \big)'.
\]
From condition~\ref{condition_estim_limit_multi_d_2} of Assumption~\ref{hypo_EQR_limit_multi_d_2}, $T\E_\theta (\widehat{\theta}_T - \theta)^{\times 2}$ converges as $T\to\infty$ hence it is bounded. Let
\[
C = \sup_{T\geqslant 0} \| T \E_\theta (\widehat{\theta}_T - \theta)^{\times 2} \|_{\M_d},
\]
then, for $\mu_\theta$-almost all $x$,
\[
| y' T \big(\J_\theta r(x ,\theta)\big) \E_\theta(\widehat{\theta}_T-\theta)^{\times 2} \big(\J_\theta r(x ,\theta)\big)' y |
\leqslant \|y\|_{\R^k}^2 C g_\theta(x).
\]
Yet $\int_E g_\theta(x) d\mu_\theta(x) < \infty$, hence Lebesgue's dominated convergence theorem implies
\[
\lim_{T\to\infty} y' A_T y = y' \int_E \J_\theta r(x ,\theta) V(\theta) \big(\J_\theta r(x ,\theta)\big)'  d\mu_\theta(x) y.
\]
This is true for all $y\in\R^k$ hence
\[
\lim_{T\to\infty} A_T = \int_E \J_\theta r(x ,\theta) V(\theta) \J_\theta r(x ,\theta)\big)'  d\mu_\theta(x).
\]
So condition~\ref{condition_eqr_limit_multi_d} of Assumption~\ref{hypo_EQR_limit_multi_d} is fulfilled with
\[
R(\theta) = \int_E \J_\theta r(x ,\theta) V(\theta) \big(\J_\theta r(x ,\theta)\big)' d\mu_\theta(x).
\]
Now applying Proposition~\ref{EQR_limit_multi_d}, we deduce the result.
\end{proof}

\begin{remarque}
When Proposition~\ref{EQR_limit_multi_d_2} applies, the asymptotic efficiency of a plug-in estimator $r(\cdot,\widehat{\theta}_T)$, or a plug-in predictor $r(X_T,\widehat{\theta}_T)$, comes down to the asymptotic efficiency of the estimator $\widehat{\theta}_T$.
More precisely, if Assumption~\ref{hypo_EQR_limit} is satisfied, the family is LAN and $I(\theta)$ is the asymptotic Fisher information, then $r(\cdot,\widehat{\theta}_T)$ is asymptotically efficient iff
\[
\lim_{T\to\infty} T \E_\theta (\widehat{\theta}_T - \theta)^{\times 2} = I(\theta)^{-1}.
\]

If in addition Assumptions~\ref{hypo_beta_tilde} and \ref{hypo_theta_T_theta_S} are satisfied, then
\[
\lim_{T\to\infty} T R_T(\theta) = R(\theta),
\]
and a plug-in predictor $r(X_T,\widehat{\theta}_T)$ is asymptotically efficient iff
\[
\lim_{T\to\infty} T \E_\theta (\widehat{\theta}_T - \theta)^{\times 2} = I(\theta)^{-1}.
\]
\end{remarque}

\subsubsection{Bivariate Ornstein-Uhlenbeck process}

We are going to see that Assumptions~\ref{hypo_EQR_limit_multi_d} and \ref{hypo_EQR_limit_multi_d_2} are satisfied for the problem of forecasting of the bivariate stationary Ornstein-Uhlenbeck process seen earlier.

Let $\theta,\,\theta'\in\Theta$, and
\[
M(x) = \begin{pmatrix} x_1 & x_2 \\ x_2 & x_1 \end{pmatrix},
\]
then
\begin{align*}
\|\J_\theta r(x,\theta')-\J_\theta r(x,\theta)\|_{\M_2}^2
&= h^2 \| (e^{-hQ(\theta')} - e^{-hQ(\theta)}) M(x) \|_{\M_2}^2\\
&= h^2 \Big( 
\| (e^{-hQ(\theta')} - e^{-hQ(\theta)}) x \|_{\R^2}^2\\
&\quad+ \| (e^{-hQ(\theta')} - e^{-hQ(\theta)}) A x \|_{\R^2}^2
\Big)\\
&= h^2 \big( \| r(x,\theta') - r(x,\theta) \|_{\R^2}^2\\
& \quad + \| r(Ax,\theta') - r(Ax,\theta) \|_{\R^2}^2 \big)\\
& \leqslant h^2 \Big( 
\| \J_\theta r(x,\theta_1) \|_{\M_2}^2\\
& \quad \quad \quad + \| \J_\theta r(Ax,\theta_2) \|_{\M_2}^2 \Big) \| \theta' - \theta \|_{\R^2}^2
\end{align*}
with $\theta_1 = \lambda_1 \theta + (1 - \lambda_1) \theta^*$ and $\theta_2 = \lambda_2 \theta + (1 - \lambda_2) \theta^*$ and $\lambda_1,\,\lambda_2\in[0,1]$. Let
\[
c(x) = \sqrt{ \| \J_\theta r(x,\theta_1) \|_{\M_2}^2 + \| \J_\theta r(Ax,\theta_2) \|_{\M_2}^2 }.
\]
Then
\begin{align*}
\|c\|_{\mu_{\theta}}^2 
& =
\big\| \J_\theta r(\,\cdot\,,\theta_1) \big\|_{\mu_{\theta}}^2
+ \big\| \J_\theta r(\,\cdot\,,\theta_2) \big\|_{\mu_{\theta}}^2 \\
& = h^2 \, \left(
\trace \left( Q(\theta)^{-1} e^{-2hQ(\theta_1)} \right)
+ \trace \left( Q(\theta)^{-1} e^{-2hQ(\theta_2)} \right)
\right)
< \infty
\end{align*}
Hence condition~\ref{condition_c_T_multi_d} of Assumption~\ref{hypo_EQR_limit_multi_d} is fulfilled with $\alpha=1$. The other conditions of Assumption~\ref{hypo_EQR_limit_multi_d} are verified applying Theorem~2.8 p.121 \cite{kutoyants2004}. Condition~2 of Assumption~\ref{hypo_EQR_limit_multi_d_2} is fulfilled because $\mu_{\theta}$ and $r$ do not depend on $T$ and $\E_\theta \left\| \J_\theta r(X_T,\theta) \right\|_{\M_2} < \infty$. Condition~3 is fulfilled because $\lim_{T\to\infty} T \E_\theta (\widehat{\theta}_T - \theta)^{\times 2} = I(\theta)^{-1}$ (Theorem~2.8 p.121 of \cite{kutoyants2004}). The other conditions of Assumptions~\ref{hypo_EQR_limit_multi_d_2} are shared with Assumption~\ref{hypo_EQR_limit_multi_d}. So we can apply Theorems~\ref{EQR_limit_multi_d} and \ref{EQR_limit_multi_d_2} to this problem, and we find again
\[
\lim_{T\to\infty} T \rho_T(\theta) = \E_\theta \left( U I(\theta)^{-1} U' \right),
\]
with $U = \J_\theta r(X_0,\theta)$.

\section*{Acknowledgements}
I wish to thank Professors Denis Bosq and Delphine Blanke for their helpful comments during the preparation of this paper. I thank an anonymous referee for helpful comments that made possible to improve the article.

\bibliographystyle{imsart-nameyear}
\bibliography{../../Biblio}

\end{document}